\documentclass[10pt,journal,twocolumn]{IEEEtran}

\usepackage[english]{babel}
\usepackage[latin1]{inputenc}
\usepackage[T1]{fontenc}
\usepackage{url}
\usepackage{graphicx}
\usepackage{subfigure}
\usepackage{color}
\usepackage{amsmath,amsthm,amsfonts,amssymb}
\usepackage{colortbl}                   % Para tabelas com sombreados
\usepackage{url}
\usepackage{booktabs}
\usepackage{epsfig}
\usepackage{pstricks}
\usepackage{pst-node}
\usepackage{pst-grad}
\usepackage{pstricks-add}
\usepackage{ifthen}
\usepackage{algorithm}
\usepackage{algpseudocode}
\usepackage{float}
\newfloat{algorithm}{t}{lop}

\newboolean{draft}
\newcommand{\isdraft}[2]{\ifthenelse{\boolean{draft}}{#1}{#2}}

\isdraft{\usepackage{setspace}}{}                              % DRAFT
\isdraft{\usepackage[footnotesize]{caption}}{}                 % DRAFT
\isdraft{\usepackage{paralist}}{}                              % DRAFT

\theoremstyle{definition}

\theoremstyle{plain}
\newtheorem{Theorem}{Theorem}
\newtheorem{Corollary}{Corollary}
\newtheorem{Lemma}{Lemma}
\newtheorem{Assumption}{Assumption}

\definecolor{red}{rgb}{0.69921875,0.28125,0.28125}
\definecolor{green}{rgb}{0.265625,0.609375,0.265625}
\definecolor{blue}{rgb}{0.3828125,0.3828125,0.69140625}
\definecolor{lightred}{rgb}{0.9140625,0.8046875,0.8046875}
\definecolor{lightblue}{rgb}{0.87109375,0.87109375,0.9296875}
\definecolor{lightgreen}{rgb}{0.796875,0.91796875,0.796875}
\definecolor{lightyellow}{RGB}{232,238,184}
\definecolor{white}{RGB}{253,253,253}

\newcommand{\mypar}[1]{{\bf #1.}}

\isdraft{\renewcommand{\baselinestretch}{1.4}}{}

\title{Distributed Optimization With Local Domains: Applications in MPC and Network Flows}

\author{Jo\~ao F.~C.~Mota, Jo\~ao M.~F.~Xavier, Pedro M.~Q.~Aguiar, and~Markus P\"uschel
\IEEEcompsocitemizethanks{\IEEEcompsocthanksitem Jo\~ao F.~C.~Mota, Jo\~ao M.~F.~Xavier, and Pedro M.~Q.~Aguiar are with Instituto de Sistemas e Rob\'otica (ISR), Instituto Superior T\'ecnico (IST), Technical University of Lisbon, Portugal.
\IEEEcompsocthanksitem Markus P{\"u}schel is with the Department of
Computer Science at ETH Zurich, Switzerland.
\IEEEcompsocthanksitem Jo\~ao F.~C.~Mota is also with the Department of Electrical and Computer Engineering at Carnegie Mellon University, USA.}%
\thanks{This work was supported by the following grants from Funda\c{c}\~{a}o para a Ci\^{e}ncia e Tecnologia (FCT): CMU-PT/SIA/0026/2009, PEst-OE/EEI/LA0009/2011, and SFRH/BD/33520/2008 (through the Carnegie Mellon/Portugal Program managed by ICTI).}}

\begin{document}

\maketitle

\begin{abstract}
	In this paper we consider a network with~\boldmath{$P$} nodes, where each node has exclusive access to a local cost function. Our contribution is a communication-efficient distributed algorithm that finds a vector~\boldmath{$x^\star$} minimizing the sum of all the functions. We make the additional assumption that the functions have intersecting local domains, i.e., each function depends only on some components of the variable. Consequently, each node is interested in knowing only some components of~\boldmath{$x^\star$}, not the entire vector. This allows for improvement in communication-efficiency. We apply our algorithm to model predictive control (MPC) and to network flow problems and show, through experiments on large networks, that our proposed algorithm requires less communications to converge than prior algorithms.
\end{abstract}


\isdraft{\pagebreak}{}

\section{Introduction}
\label{Sec:Intro}

	Distributed algorithms have become popular for solving optimization problems formulated on networks. Consider, for example, a network with~$P$ nodes and the following problem:
	\begin{equation}\label{Eq:SeparableOptim}
		\underset{x \in \mathbb{R}^n}{\text{minimize}} \,\,\, f_1(x) + f_2(x) + \cdots + f_P(x)\,,
	\end{equation}
	where~$f_p$ is a function known only at node~$p$. Fig.~\ref{SubFig:GlobalDomains} illustrates this problem for a variable~$x$ of size~$n = 3$. Several algorithms have been proposed to solve~\eqref{Eq:SeparableOptim} in a distributed way, that is, each node communicates only with its neighbors and there is no central node. In a typical distributed algorithm for~\eqref{Eq:SeparableOptim}, each node holds an estimate of a solution~$x^\star$, and iteratively updates and exchanges it with its neighbors. It is usually assumed that all nodes are interested in knowing the entire solution~$x^\star$. While such an assumption holds for problems like consensus~\cite{Rabbat04-DistributedOptimizationSensorNetworks} or distributed SVMs~\cite{Forero10-ConsensusBasedDistributedSVMs}, there are important problems where it does not hold, especially in the context of large networks. Two examples we will explore here are distributed Model Predictive Control (MPC) and network flows. The goal in distributed MPC is to control a network of interacting subsystems with coupled dynamics~\cite{Camponogara02-DistributedMPC}. That control should be performed using the least amount of energy. Network flow problems have many applications~\cite{Ahuja93-NetworkFlows}; here, we will solve a network flow problem to minimize delays in a multicommodity routing problem. Both distributed MPC and network flow problems can be written naturally as~\eqref{Eq:SeparableOptim} with functions that depend only on a subset of the components of~$x$.

	\begin{figure}[t]
  \centering
  \subfigure[Global variable]{\label{SubFig:GlobalDomains}
    \psscalebox{0.865}{
      \begin{pspicture}(4.5,4.8)
        \def\nodesimp{
          \pscircle*[linecolor=black!65!white](0,0){0.3}
        }

        \rput(0.4,4.1){\rnode{C1}{\nodesimp}}   \rput(0.4,4.1){\small \textcolor{white}{$1$}}
        \rput(0.9,2.6){\rnode{C2}{\nodesimp}}   \rput(0.9,2.6){\small \textcolor{white}{$2$}}
        \rput(0.8,0.9){\rnode{C3}{\nodesimp}}   \rput(0.8,0.9){\small \textcolor{white}{$3$}}
        \rput(3.0,1.1){\rnode{C4}{\nodesimp}}   \rput(3.0,1.1){\small \textcolor{white}{$4$}}
        \rput(4.0,2.6){\rnode{C5}{\nodesimp}}   \rput(4.0,2.6){\small \textcolor{white}{$5$}}
        \rput(2.2,3.3){\rnode{C6}{\nodesimp}}   \rput(2.2,3.3){\small \textcolor{white}{$6$}}

        \ncline[nodesep=0.33cm,linewidth=0.9pt]{-}{C1}{C2}
        \ncline[nodesep=0.33cm,linewidth=0.9pt]{-}{C1}{C6}
        \ncline[nodesep=0.33cm,linewidth=0.9pt]{-}{C2}{C3}
        \ncline[nodesep=0.33cm,linewidth=0.9pt]{-}{C2}{C6}
        \ncline[nodesep=0.33cm,linewidth=0.9pt]{-}{C3}{C4}
        \ncline[nodesep=0.33cm,linewidth=0.9pt]{-}{C4}{C5}
        \ncline[nodesep=0.33cm,linewidth=0.9pt]{-}{C5}{C6}

        \rput[lb](-0.11,4.45){ $f_1(x_1,x_2,x_3)$}
        \rput[lt](1.1,2.4){$f_2(x_1,x_2,x_3)$}
        \rput[lt](0.1,0.56){$f_3(x_1,x_2,x_3)$}
        \rput[t](3.32,0.75){$f_4(x_1,x_2,x_3)$}
        \rput[b](4.05,3.00){$f_5(x_1,x_2,x_3)$}
        \rput[b](2.52,3.68){$f_6(x_1,x_2,x_3)$}

        %\psgrid
      \end{pspicture}
    }
  }
  \hfill
  \subfigure[Partial variable]{\label{SubFig:PartialDomains}
    \psscalebox{0.865}{
      \begin{pspicture}(4.5,4.8)
        \def\nodesimp{
          \pscircle*[linecolor=black!65!white](0,0){0.3}
        }

        \rput(0.4,4.1){\rnode{C1}{\nodesimp}}   \rput(0.4,4.1){\small \textcolor{white}{$1$}}
        \rput(0.9,2.6){\rnode{C2}{\nodesimp}}   \rput(0.9,2.6){\small \textcolor{white}{$2$}}
        \rput(0.8,0.9){\rnode{C3}{\nodesimp}}   \rput(0.8,0.9){\small \textcolor{white}{$3$}}
        \rput(3.0,1.1){\rnode{C4}{\nodesimp}}   \rput(3.0,1.1){\small \textcolor{white}{$4$}}
        \rput(4.0,2.6){\rnode{C5}{\nodesimp}}   \rput(4.0,2.6){\small \textcolor{white}{$5$}}
        \rput(2.2,3.3){\rnode{C6}{\nodesimp}}   \rput(2.2,3.3){\small \textcolor{white}{$6$}}

        \ncline[nodesep=0.33cm,linewidth=0.9pt]{-}{C1}{C2}
        \ncline[nodesep=0.33cm,linewidth=0.9pt]{-}{C1}{C6}
        \ncline[nodesep=0.33cm,linewidth=0.9pt]{-}{C2}{C3}
        \ncline[nodesep=0.33cm,linewidth=0.9pt]{-}{C2}{C6}
        \ncline[nodesep=0.33cm,linewidth=0.9pt]{-}{C3}{C4}
        \ncline[nodesep=0.33cm,linewidth=0.9pt]{-}{C4}{C5}
        \ncline[nodesep=0.33cm,linewidth=0.9pt]{-}{C5}{C6}

        \rput[lb](0.0,4.45){ $f_1(x_1,x_2)$}
        \rput[lt](1.1,2.4){$f_2(x_1,x_2,x_3)$}
        \rput[lt](0.2,0.56){$f_3(x_2,x_3)$}
        \rput[t](3.0,0.74){$f_4(x_1,x_3)$}
        \rput[b](3.9,2.97){$f_5(x_2,x_3)$}
        \rput[b](2.6,3.65){$f_6(x_1,x_2)$}

        %\psgrid
      \end{pspicture}
    }
  }
  \caption{
		Example of a $\text{(a)}$ global and a $\text{(b)}$ partial variable. While each function in $\text{(a)}$ depends on all the components of the variable $x = (x_1,x_2,x_3)$, each function in $\text{(b)}$ depends only on a subset of the components of $x$.
  }
  \label{Fig:IntroExamplesVariable}
  \end{figure}
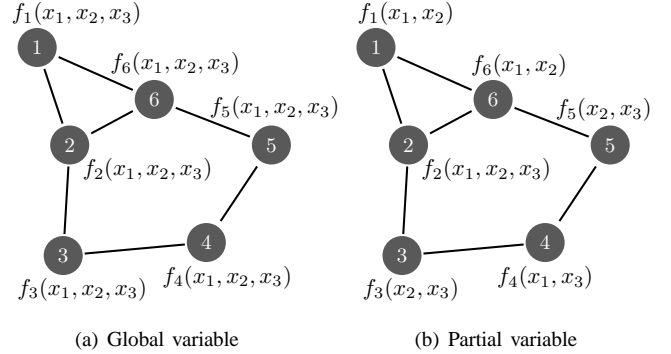

	We solve~\eqref{Eq:SeparableOptim} in the case that each function~$f_p$ may depend only on a subset of the components of the variable $x \in \mathbb{R}^n$. This situation is illustrated in Fig.~\ref{SubFig:PartialDomains}, where, for example, $f_1$ only depends on~$x_1$ and~$x_2$. To capture these dependencies, we write~$x_S$, $S \subseteq \{1,\ldots,n\}$, to denote a subset of the components of~$x$. For example, if~$S = \{2,4\}$, then $x_{S}=(x_2,x_4)$. With this notation, our goal is solving
	\begin{equation}\label{Eq:PartialProb}
		\underset{x \in \mathbb{R}^n}{\text{minimize}} \,\,\, f_1(x_{S_1}) + f_2(x_{S_2}) + \cdots + f_P(x_{S_P})\,,
	\end{equation}
	where~$S_p$ is the set of components the function~$f_p$ depends on. Accordingly, every node~$p$ is only interested in a part of the solution: $x^\star_{S_p}$. We make the following
	\begin{Assumption}\label{Ass:PartialVariable}
		No components are global, i.e., $\bigcap_{p=1}^P S_p = \emptyset$.
	\end{Assumption}
	Whenever this assumption holds, we say that the variable~$x$ is \textit{partial}. Fig.~\ref{SubFig:PartialDomains} shows an example of a partial variable. Note that, although no component appears in all nodes, node~$2$ depends on all components, i.e., it has a global domain. In fact, Assumption~\ref{Ass:PartialVariable} allows only a strict subset of nodes to have global domains. This contrasts with Fig.~\ref{SubFig:GlobalDomains}, where all nodes have global domains and hence Assumption~\ref{Ass:PartialVariable} does not hold. We say that the variable~$x$ in Fig.~\ref{SubFig:GlobalDomains} is \textit{global}. Clearly, problem~\eqref{Eq:PartialProb} is a particular case of problem~\eqref{Eq:SeparableOptim} and hence it can be solved with any algorithm designed for~\eqref{Eq:SeparableOptim}. This approach, however, may introduce unnecessary communications, since nodes exchange full estimates of~$x^\star$, and not just of the components they are interested in, thus potentially wasting useful communication resources.	In many networks, communication is the operation that consumes the most energy and/or time.

	\mypar{Contributions}
	We first formalize problem~\eqref{Eq:PartialProb} by making a clear distinction between variable dependencies and communication network. Before, both were usually assumed the same. Then, we propose a distributed algorithm for problem~\eqref{Eq:PartialProb} that takes advantage of its special structure to reduce communications. We will distinguish two cases for the variable of~\eqref{Eq:PartialProb}: connected and non-connected, and design algorithms for both. To our knowledge, this is the first time an algorithm has been proposed for a non-connected variable. We apply our algorithms to distributed MPC and to network flow problems. A surprising result is that, despite their generality, the proposed algorithms outperform prior algorithms even though they are application-specific.

% 	We propose a communication-efficient algorithm for~\eqref{Eq:PartialProb} in a two-step approach. First, we consider a particular case of~\eqref{Eq:PartialProb}, for which an algorithm is derived. Then, we generalize that algorithm to solve any instance of~\eqref{Eq:PartialProb}. To see what the particular case is, consider component~$x_1$ in Fig.~\ref{SubFig:PartialDomains}. The nodes that depend on it, i.e., nodes~$1$, $2$, $4$, and~$6$, form two isolated components: nodes~$1$, $2$, and~$6$ can communicate directly among themselves, but none of them communicates with node~$4$. Since all these nodes need to agree on an optimal value for~$x_1$, we need to make them communicate. So, either node~$3$ or node~$5$ needs to ``retransmit'' estimates of~$x_1$, a component it is not interested in. As a first step, we will assume this situation does not occur, i.e., that all the nodes depending on the same component form a connected subgraph. This is our particular case, and we name a variable satisfying it a \textit{connected variable}. In Fig.~\ref{SubFig:PartialDomains}, components~$x_2$ and~$x_3$ induce connected subgraphs, but~$x_1$ does not. The variable is said \textit{non-connected} in this case.
%
% 	Before stating our problem formally in Section~\ref{Sec:ProbStat}, we overview related work and state our contributions.

	\mypar{Related work}
	Many algorithms have been proposed for the global problem~\eqref{Eq:SeparableOptim}, for example, gradient-based methods~\cite{Rabbat04-DistributedOptimizationSensorNetworks,CXVOptimInSPandComm,Jakovetic11-FastDistributedMethods}, or methods based on the \textit{Alternating Direction Method of Multipliers} (ADMM)~\cite{Zhu09-DistributedInNetworkChannelCoding,Schizas08-ConsensusAdHocWSNsPartI,Mota12-DADMM}. As mentioned before, solving~\eqref{Eq:PartialProb} with an algorithm designed for~\eqref{Eq:SeparableOptim} introduces unnecessary communications. We will observe this when we compare the algorithm proposed here with D-ADMM~\cite{Mota12-DADMM}, the state-of-the-art for~\eqref{Eq:SeparableOptim} in terms of communication-efficiency.

	To our knowledge, this is the first time problem~\eqref{Eq:PartialProb} has been explicitly stated in a distributed context. For example, \cite[\S7.2]{Boyd11-ADMM} proposes an algorithm for~\eqref{Eq:PartialProb}, but is not distributed in our sense. Namely, it either requires a platform that supports all-to-all communications (in other words, a central node), or requires running consensus algorithms on each induced subgraph, at each iteration~\cite[\S10.1]{Boyd11-ADMM}. Thus, that algorithm is only distributed when every component induces subgraphs that are stars. Actually, we found only one algorithm in the literature that is distributed (or that can easily be made distributed) for all the scenarios considered in this paper. That algorithm was proposed in~\cite{Kekatos12-DistributedRobustPowerStateEstimation} in the context of power system state estimation (the algorithm we propose can also be applied to this problem, although we will not consider it here). Our simulations show that the algorithm in~\cite{Kekatos12-DistributedRobustPowerStateEstimation} requires always more communications than the algorithm we propose.

	Although we found just one (communication-efficient) distributed algorithm solving~\eqref{Eq:PartialProb}, there are many other algorithms solving particular instances of it. For example, in network flow problems, each component of the variable is associated to an edge of the network. We will see such problems can be written as~\eqref{Eq:PartialProb} with a connected variable, in the special case where each induced subgraph is a star. In this case, \cite[\S7.2]{Boyd11-ADMM} becomes distributed, and also gradient/subgradient methods can be applied directly either to the primal problem~\cite{Tsitsiklis84-DistributedAsynchronousOptimalRouting} or to the dual problem~\cite{Zargham12-AcceleratedDualDescent}, and yield distributed algorithms. Network flow problems have also been tackled with Newton-like methods~\cite{Bertsekas83-ProjectedNewtonMethods,Zargham12-AcceleratedDualDescent}. A related problem is Network Utility Maximization (NUM), which is used to model traffic control on the Internet~\cite{Kelly97-ChargingRateControlElasticTraffic,Low02-UnderstandingVegas}. For example, the TCP/IP protocol has been interpreted as a gradient algorithm solving a NUM. In~\cite{Mota12-DistributedMPC}, we compared a particular instance of the proposed algorithm with prior algorithms solving NUM, and showed that it requires less end-to-end communications. However, due to its structure, it does not offer interpretations of end-to-end protocols as realistic as gradient-based algorithms.

	Distributed Model Predictive Control (MPC)~\cite{Camponogara02-DistributedMPC} is another problem that has been addressed with algorithms solving~\eqref{Eq:PartialProb}, again in the special case of a variable whose components induce star subgraphs only. Such algorithms include subgradient methods~\cite{Wakasa08-DecentralizedMPCDualDecomp}, interior-point methods~\cite{Camponogara11-DistributedMPC}, fast gradient~\cite{Conte12-ComputationalAspectsDistributedMPC}, and ADMM-based methods~\cite{Conte12-ComputationalAspectsDistributedMPC,Summers12-DistributedMPConsensus} (which apply~\cite[\S7.2]{Boyd11-ADMM}). All these methods were designed for the special case of star-shaped induced subgraphs and, similarly to~\cite[\S7.2]{Boyd11-ADMM}, they become inefficient if applied to more generic cases. In spite of its generality, the algorithm we propose requires less communications than previous algorithms that were specifically designed for distributed MPC or network flow problems.

	Additionally, we apply our algorithm to two scenarios in distributed MPC that have not been considered before: problems where the variable is connected but the induced subgraphs are not stars, and problems with a non-connected variable. Both cases can model scenarios where subsystems that are coupled through their dynamics cannot communicate directly.

	Lastly, this paper extends considerably our preliminary work~\cite{Mota12-DistributedMPC}. In particular, the algorithm in~\cite{Mota12-DistributedMPC} was designed for bipartite networks and was based on the $2$-block ADMM. In contrast, the algorithms proposed here work on any connected network and are based on the Extended ADMM; thus, they have different convergence guarantees. Also, the MPC model proposed here is significantly more general than the one in~\cite{Mota12-DistributedMPC}.

\section{Terminology and Problem Statement}
\label{Sec:ProbStat}

	We start by introducing the concepts of \textit{communication network} and \textit{variable connectivity}.

\mypar{Communication network}
	 A communication network is represented as an undirected graph~$\mathcal{G} = (\mathcal{V},\mathcal{E})$, where $\mathcal{V} = \{1,\ldots,P\}$ is the set of nodes and~$\mathcal{E} \subseteq \mathcal{V} \times \mathcal{V}$ is the set of edges. Two nodes communicate directly if there is an edge connecting them in~$\mathcal{G}$. We assume:
	\begin{Assumption}\label{Ass:Network}
		$\mathcal{G}$ is connected and its topology does not change over time; also, a coloring scheme~$\mathcal{C}$ of~$\mathcal{G}$ is available beforehand.
	\end{Assumption}
	A coloring scheme~$\mathcal{C}$ is a set of numbers, called colors, assigned to the nodes such that two neighbors never have the same color, as shown in Fig.~\ref{Fig:ColoringScheme}. Given its importance in TDMA, a widespread protocol for avoiding packet collisions, there is a large literature on coloring networks, as briefly overviewed in~\cite{Mota12-DistributedBP}.	Our algorithm integrates naturally with TDMA, since both use coloring as a synchronization scheme: nodes work sequentially according to their colors, and nodes with the same color work in parallel. The difference is that TDMA uses a more restrictive coloring, as nodes within two hops cannot have the same color. Note that packet collision is often ignored in the design of distributed algorithms, as confirmed by the ubiquitous assumption that all nodes can communicate simultaneously.

	We associate with each node~$p$ in the communication network a function~$f_p: \mathbb{R}^{n_p} \xrightarrow{} \mathbb{R} \cup \{+\infty\}$, where~$n = n_1 + \cdots + n_P$, and make the
	\begin{Assumption}\label{Ass:Function}
		Each function~$f_p$ is known only at node~$p$ and it is closed, proper, and convex over~$\mathbb{R}^{n_p}$.
	\end{Assumption}
	Since we allow~$f_p$ to take infinite values, each node can impose constraints on the variable using indicator functions, i.e., functions that evaluate to $+\infty$ when the constraints are not satisfied, and to~$0$ otherwise.

	\begin{figure}[t]
		\centering
    \psscalebox{0.865}{
      \begin{pspicture}(4.5,4.8)
        \def\nodesimp{
					\pscircle*[linecolor=black!65!white](0,0){0.3}
        }
        \def\nodeB{
            \pscircle*[linecolor=black!15!white](0,0){0.3}
        }
        \def\nodeC{
            \pscircle[fillstyle=vlines*,linecolor=black!50!white,hatchcolor=black!50!white](0,0){0.3}
        }

        \rput(0.4,4.1){\rnode{C1}{\nodesimp}}   \rput(0.4,4.1){\small \textcolor{white}{$1$}}
        \rput(0.9,2.6){\rnode{C2}{\nodeC}}      \rput(0.9,2.6){\small \textcolor{black}{$2$}}
        \rput(0.8,0.9){\rnode{C3}{\nodesimp}}   \rput(0.8,0.9){\small \textcolor{white}{$3$}}
        \rput(3.0,1.1){\rnode{C4}{\nodeB}}      \rput(3.0,1.1){\small \textcolor{black}{$4$}}
        \rput(4.0,2.6){\rnode{C5}{\nodesimp}}   \rput(4.0,2.6){\small \textcolor{white}{$5$}}
        \rput(2.2,3.3){\rnode{C6}{\nodeB}}      \rput(2.2,3.3){\small \textcolor{black}{$6$}}

        \ncline[nodesep=0.33cm,linewidth=0.9pt]{-}{C1}{C2}
        \ncline[nodesep=0.33cm,linewidth=0.9pt]{-}{C1}{C6}
        \ncline[nodesep=0.33cm,linewidth=0.9pt]{-}{C2}{C3}
        \ncline[nodesep=0.33cm,linewidth=0.9pt]{-}{C2}{C6}
        \ncline[nodesep=0.33cm,linewidth=0.9pt]{-}{C3}{C4}
        \ncline[nodesep=0.33cm,linewidth=0.9pt]{-}{C4}{C5}
        \ncline[nodesep=0.33cm,linewidth=0.9pt]{-}{C5}{C6}

%         \rput[rt](0.117157,3.817157){\small $1$}
%         \rput[rb](0.617157,2.882843){\small $3$}
%         \rput[rb](0.517157,1.182843){\small $1$}
%         \rput[rb](2.717157,1.382843){\small $2$}
%         \rput[lt](4.282843,2.317157){\small $1$}
%         \rput[lt](2.482843,3.017157){\small $2$}

        \rput[lb](0.45,4.48){\small $1$}
        \rput[lt](1.18,2.40){\small $3$}
        \rput[lt](0.6,0.54){\small $1$}
        \rput[t](3.10,0.74){\small $2$}
        \rput[b](4.09,2.97){\small $1$}
        \rput[b](2.30,3.66){\small $2$}

        %\psgrid
      \end{pspicture}
			}
  \caption{
		Example of a coloring scheme of the communication network using~$3$ colors: $\mathcal{C}_1  = \{1,3,5\}$,	$\mathcal{C}_2 = \{4,6\}$, and~$\mathcal{C}_3 = \{2\}$.
  }
  \label{Fig:ColoringScheme}
  \end{figure}
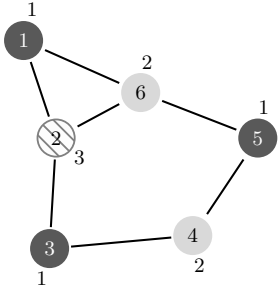

\mypar{Variable connectivity}
	Although each function~$f_p$ is available only at node~$p$, each component of the variable~$x$ may be associated with several nodes. Let~$x_l$ be a given component. The \textit{subgraph induced by}~$x_l$ is represented by~$\mathcal{G}_l = (\mathcal{V}_l, \mathcal{E}_l) \subseteq \mathcal{G}$, where~$\mathcal{V}_l$ is the set of nodes whose functions depend on~$x_l$, and an edge~$(i,j) \in \mathcal{E}$ belongs to~$\mathcal{E}_l$ if both~$i$ and~$j$ are in~$\mathcal{V}_l$. For example, the subgraph induced by~$x_1$ in Fig.~\ref{SubFig:PartialDomains} consists of $\mathcal{V}_1 = \{1,2,4,6\}$ and~$\mathcal{E}_1 = \{(1,2),(1,6),(2,6)\}$. We say that~$x_l$ is \textit{connected} if its induced subgraph is connected, and \textit{non-connected} otherwise. Likewise, a \textit{variable is connected} if all its components are connected, and \textit{non-connected} if it has at least one non-connected component.

\mypar{Problem statement}
	Given a network satisfying Assumption~\ref{Ass:Network} and a set of functions satisfying Assumptions~\ref{Ass:PartialVariable} and~\ref{Ass:Function}, we solve the following problem: \textit{design a distributed, communication-efficient algorithm that solves~\eqref{Eq:PartialProb}, either with a connected or with a non-connected variable.}

	By distributed algorithm we mean a procedure that makes no use of a central node and where each node communicates only with its neighbors. Unfortunately there is no known lower bound on how many communications are needed to solve~\eqref{Eq:PartialProb}. Because of this, communication-efficiency can only be assessed relative to existing algorithms that solve the same problem. As mentioned before, our strategy for this problem is to design an algorithm for the connected case and then generalize it to the non-connected case.

\section{Connected Case}
\label{Sec:ConnectedCase}

	In this section we derive an algorithm for~\eqref{Eq:PartialProb} assuming its variable is connected. Our derivation uses the same principles as the state-of-the-art algorithm~\cite{Mota12-DADMM,Mota12-DistributedBP} for the global problem~\eqref{Eq:SeparableOptim}. The main idea is to manipulate~\eqref{Eq:PartialProb} to make the Extended ADMM~\cite{Han12-NoteOnADMM} applicable. We will see that the algorithm derived here generalizes the one in~\cite{Mota12-DADMM,Mota12-DistributedBP}.

\mypar{Problem manipulation}
	Let~$x_l$ be a given component and~$\mathcal{G}_l = (\mathcal{V}_l,\mathcal{E}_l)$ the respective induced subgraph. In this section we assume each~$\mathcal{G}_l$ is connected. Since all nodes in~$\mathcal{V}_l$ are interested in~$x_l$, we will create a copy of~$x_l$ in each of those nodes: $x_l^{(p)}$ will be the copy at node~$p$ and~$x_{S_p}^{(p)} := \{x_l^{(p)}\}_{l \in S_p}$ will be the set of all copies at node~$p$. We rewrite~\eqref{Eq:PartialProb} as
	\begin{equation}\label{Eq:PartialManip1}
		\begin{array}{ll}
			\underset{\{\bar{x}_l\}_{l=1}^n}{\text{minimize}} & f_1(x_{S_1}^{(1)}) + f_2(x_{S_2}^{(2)}) + \cdots + f_P(x_{S_P}^{(P)}) \\
			\text{subject to} & x_l^{(i)} = x_l^{(j)}\,, \quad (i,j) \in \mathcal{E}_l\,,\,\,\, l = 1,\ldots,n\,,
		\end{array}
	\end{equation}
	where~$\{\bar{x}_l\}_{l=1}^L$ is the optimization variable and represents the set of all copies. We used~$\bar{x}_l$ to denote all copies of the component~$x_l$, which are located only in the nodes of~$\mathcal{G}_l$: $\bar{x}_l := \{x_l^{(p)}\}_{p \in \mathcal{V}_l}$. The reason for introducing constraints in~\eqref{Eq:PartialManip1} is to enforce equality among the copies of the same component: if two neighboring nodes~$i$ and~$j$ depend on~$x_l$, then~$x_l^{(i)} = x_l^{(j)}$ appears in the constraints of~\eqref{Eq:PartialManip1}. We assume that any edge in the communication network is represented as the ordered pair~$(i,j) \in \mathcal{E}$, with~$i<j$. As such, there are no repeated equations in~\eqref{Eq:PartialManip1}.
	Problems~\eqref{Eq:PartialProb} and~\eqref{Eq:PartialManip1} are equivalent because each induced subgraph is connected.

	A useful observation is that~$x_l^{(i)} = x_l^{(j)}$, $(i,j) \in \mathcal{E}_l$, can be written as~$A_l \bar{x}_l = 0$, where~$A_l$ is the transposed node-arc incidence matrix of the subgraph~$\mathcal{G}_l$. The node-arc incidence matrix represents a given graph with a matrix where each column corresponds to an edge $(i,j) \in \mathcal{E}$ and has~$1$ in the $i$th entry, $-1$ in the $j$th entry, and zeros elsewhere. We now partition the optimization variable according to the coloring scheme: for each~$l = 1,\ldots,n$, $\bar{x}_l = (\bar{x}_l^1,\ldots,\bar{x}_l^C)$, where
	$$
		\bar{x}_l^c =
		\left\{
			\begin{array}{ll}
				\{x_l^{(p)}\}_{p \in \mathcal{V}_l \cap \mathcal{C}_c}\,, &\quad \text{if $\mathcal{V}_l \cap \mathcal{C}_c \neq \emptyset$} \\
				\emptyset\,, &\quad \text{if $\mathcal{V}_l \cap \mathcal{C}_c = \emptyset$}
			\end{array}\,,
		\right.
	$$
	and~$\mathcal{C}_c$ is the set of nodes that have color~$c$. Thus, $\bar{x}_l^c$ is the set of copies of~$x_l$ held by the nodes that have color~$c$. If no node with color~$c$ depends on~$x_l$, then $\bar{x}_l^c$ is empty. A similar notation for the columns of the matrix~$A_l$ enables us to write $A_l \bar{x}_l$ as $\bar{A}_l^1 \bar{x}_l^1 + \cdots + \bar{A}_l^C \bar{x}_l^C$, and thus~\eqref{Eq:PartialManip1} equivalently as
  \begin{equation}\label{Eq:PartialManip2}
  	\begin{array}{ll}
			\underset{\bar{x}^1,\ldots,\bar{x}^C}{\text{minimize}} & \sum_{p \in \mathcal{C}_1} f_p(x_{S_p}^{(p)}) +  \cdots + \sum_{p \in \mathcal{C}_C} f_p(x_{S_p}^{(p)})\\
			\text{subject to} & \bar{A}^1 \bar{x}^1 + \cdots + \bar{A}^C \bar{x}^C = 0\,,
		\end{array}
  \end{equation}
  where~$\bar{x}^c = \{\bar{x}_l^c\}_{l=1}^n$, and~$\bar{A}^c$ is the diagonal concatenation of the matrices~$\bar{A}_1^c$, $\bar{A}_2^c$, \ldots, $\bar{A}_n^c$, i.e., $\bar{A}^c= \text{diag}(\bar{A}_1^c,\bar{A}_2^c,\ldots,\bar{A}_n^c)$. To better visualize the constraint in~\eqref{Eq:PartialManip2}, we wrote
  \begin{equation}\label{Eq:PartialStructureMatrices}
    \bar{A}^c \bar{x}^c
    =
		\begin{bmatrix}
			\bar{A}_1^c &              &        &             \\
			            & \bar{A}_2^c  &        &             \\
			            &              & \ddots &             \\
			            &              &        & \bar{A}_n^c
		\end{bmatrix}
		\begin{bmatrix}
			\bar{x}_1^c \\
			\bar{x}_2^c \\
			\vdots \\
			\bar{x}_n^c
		\end{bmatrix}
  \end{equation}
  for each~$c = 1,\ldots,C$.
  The format of~\eqref{Eq:PartialManip2} is exactly the one to which the Extended ADMM applies, as explained next.

\mypar{Extended ADMM}
	The Extended ADMM is a natural generalization of the \textit{Alternating Direction Method of Multipliers} (ADMM). Given a set of closed, convex functions~$g_1$, $\ldots$, $g_C$, and a set of full column rank matrices~$E_1$, \ldots, $E_C$, all with the same number of rows, the Extended ADMM solves
	\begin{equation}\label{Eq:ProblemSolvedByExtADMM}
		\begin{array}{ll}
			\underset{x_1,\ldots,x_C}{\text{minimize}} & g_1(x_1) + \cdots + g_C(x_C) \\
			\text{subject to} & E_1 x_1 + \cdots + E_C x_C = 0\,.
		\end{array}
	\end{equation}
	It consists of iterating on~$k$ the following equations:
	\begin{align}
        x_1^{k+1} &= \underset{x_1}{\arg\min} \,\,\, L_{\rho}(x_1, x_2^{k},\ldots,x_P^k;\lambda^k)
        \label{Eq:ExtADMMAlg1}
        \\
        x_2^{k+1} &= \underset{x_2}{\arg\min} \,\,\, L_{\rho}(x_1^{k+1}, x_2, x_3^k, \ldots, x_C^k; \lambda^k)
        \label{Eq:ExtADMMAlg2}
        \\
        &\,\,\,\vdots
        \notag
        \\
        x_C^{k+1} &= \underset{x_C}{\arg\min} \,\,\, L_{\rho}(x_1^{k+1}, x_2^{k+1}, \ldots, x_{C-1}^{k+1}, x_C; \lambda^k)
        \label{Eq:ExtADMMAlg3}
        \\
        \lambda^{k+1} &= \lambda^k + \rho \sum_{c = 1}^C E_c x_c^{k+1}\,,
        \label{Eq:ExtADMMAlg4}
  \end{align}
  where~$\lambda$ is the dual variable, $\rho$ is a positive parameter, and
  $$
    L_\rho(x;\lambda) = \sum_{c=1}^C \bigl(g_c(x_c) + \lambda^\top E_c x_c\bigr) + \frac{\rho}{2}\bigl\|\sum_{c=1}^C E_cx_c\bigr\|^2
  $$
  is the augmented Lagrangian of~\eqref{Eq:ProblemSolvedByExtADMM}. The original ADMM is recovered whenever~$C = 2$, i.e., when there are only two terms in the sums of~\eqref{Eq:ProblemSolvedByExtADMM}. The following theorem gathers some known convergence results for~\eqref{Eq:ExtADMMAlg1}-\eqref{Eq:ExtADMMAlg4}.

  \begin{Theorem}[\cite{Han12-NoteOnADMM,Mota11-ADMMProof}]\label{Teo:ConvergenceADMM}
    For each~$c=1,\ldots,C$, let~$g_c:\mathbb{R}^{n_c}\xrightarrow{} \mathbb{R} \cup \{+\infty\}$ be closed and convex over~$\mathbb{R}^{n_c}$ and $\text{dom}\,g_c \neq \emptyset$. Let each~$E_c$ be an $m\times n_c$ matrix. Assume~\eqref{Eq:ProblemSolvedByExtADMM} is solvable and that either $1)$ $C=2$ and each~$E_c$ has full column rank, or $2)$ $C\geq2$ and each~$g_c$ is strongly convex. Then, the sequence~$\{(x_1^k,\ldots,x_C^k,\lambda^k)\}$ generated by~\eqref{Eq:ExtADMMAlg1}-\eqref{Eq:ExtADMMAlg4} converges to a primal-dual solution of~\eqref{Eq:ProblemSolvedByExtADMM}.
  \end{Theorem}
  It is believed that \eqref{Eq:ExtADMMAlg1}-\eqref{Eq:ExtADMMAlg4} converges even when~$C > 2$, each $g_c$ is closed and convex (not necessarily strongly convex), and each matrix~$E_c$ has full column rank. Such belief is supported by empirical evidence~\cite{Mota12-DistributedBP,Han12-NoteOnADMM} and its proof remains an open problem. So far, there are only proofs for modifications of~\eqref{Eq:ExtADMMAlg1}-\eqref{Eq:ExtADMMAlg4} that resulted either in a slower algorithm~\cite{Luo12-LinearConvergenceADMM}, or in algorithms not applicable to distributed scenarios~\cite{He12-ADMGaussianBackSubstitution}.

\mypar{Applying the Extended ADMM}
	The clear correspondence between~\eqref{Eq:PartialManip2} and~\eqref{Eq:ProblemSolvedByExtADMM} makes~\eqref{Eq:ExtADMMAlg1}-\eqref{Eq:ExtADMMAlg4} directly applicable to~\eqref{Eq:PartialManip2}.
	Associate a dual variable~$\lambda_l^{ij}$ to each constraint $x_l^{(i)} = x_l^{(j)}$ in~\eqref{Eq:PartialManip1}. 	Translating~\eqref{Eq:ExtADMMAlg4} component-wise, $\lambda_l^{ij}$ is updated as
	\begin{equation}\label{Eq:PartialDualVariableUpdate}
				\lambda_l^{ij,k+1} = \lambda_l^{ij,k} + \rho \bigl(x_l^{(i),k+1} - x_l^{(j),k+1}\bigr)\,,
  \end{equation}
  where~$x_l^{(p),k+1}$ is the estimate of~$x_l$ at node~$p$ after iteration~$k$. This estimate is obtained  from~\eqref{Eq:ExtADMMAlg1}-\eqref{Eq:ExtADMMAlg3}, where we will focus our attention now. This sequence will yield the synchronization mentioned in Section~\ref{Sec:ProbStat}: nodes work sequentially according to their colors, with the same colored nodes working in parallel.
  In fact, each problem in~\eqref{Eq:ExtADMMAlg1}-\eqref{Eq:AppSecondTerm_3} corresponds to a given color. Moreover, each of these problems decomposes into $|\mathcal{C}_c|$ problems that can be solved in parallel, each by a node with color~$c$. For example, the copies of the nodes with color~$1$ are updated according to~\eqref{Eq:ExtADMMAlg1}:
	\begin{align}
		  \bar{x}^{1,k+1}
		&=
		  \underset{\bar{x}^1}{\arg\min} \,
		  \sum_{p \in \mathcal{C}_1} f_p(x_{S_p}^{(p)}) + {\lambda^k}^\top \bar{A}^1 \bar{x}^1
		  \notag
		  \\
		  &\phantom{aaaaaaaaaaaaaa} + \frac{\rho}{2}\biggl\|\bar{A}^1 \bar{x}^{1} + \sum_{c = 2}^C \bar{A}^c \bar{x}^{c,k}\biggr\|^2
		\label{Eq:PartialC1_1}
		\\
		&=
		  \underset{\bar{x}^1}{\arg\min} \,
			\sum_{p \in \mathcal{C}_1} \biggl( f_p(x_{S_p}^{(p)})
			\notag
			\\&
			+ \sum_{l \in S_p} \sum_{j \in \mathcal{N}_p \cap \mathcal{V}_l} \Bigl(\text{sign}(j-p) \lambda_l^{pj,k} - \rho\, x_l^{(j),k}\Bigr)^\top x_l^{(p)}
			\notag
			\\&\phantom{aaaaaaaaaaaaaaa}
			+ \frac{\rho}{2}\sum_{l \in S_p} D_{p,l}\Bigl(x_l^{(p)}\Bigr)^2 \biggr)\,,
		\label{Eq:PartialC1_2}
	\end{align}
	whose equivalence is established in Lemma~\ref{Lem:EquivalenceNodesC1} below. In~\eqref{Eq:PartialC1_2},
	the sign function is defined as~$1$ for nonnegative arguments and as~$-1$ for negative arguments. Also, $D_{p,l}$ is the degree of node~$p$ in the subgraph~$\mathcal{G}_l$, i.e., the number of neighbors of node~$p$ that also depend on~$x_l$. Of course, $D_{p,l}$ is only defined when~$l \in S_p$. Before establishing the equivalence between~\eqref{Eq:PartialC1_1} and~\eqref{Eq:PartialC1_2}, note that~\eqref{Eq:PartialC1_2} decomposes into~$|\mathcal{C}_1|$ problems that can be solved in parallel. This is because~$\bar{x}^1$ consists of the copies held by the nodes with color~$1$; and, since nodes with the same color are never neighbors, none of the copies in~$\bar{x}^1$ appears as~$x_l^{(j),k}$ in the second term of~\eqref{Eq:PartialC1_2}.  Therefore, all nodes~$p$ in~$\mathcal{C}_1$ can solve in parallel the following problem:
	\begin{multline}\label{Eq:PartialProbSolvedAtEachNodeC1}
		x_{S_p}^{(p),k+1}
		=
		\underset{x_{S_p}^{(p)}=\{x_l^{(p)}\}_{l \in S_p}}{\arg\min} \,
		f_p(x_{S_p}^{(p)}) \\+ \sum_{l \in S_p} \sum_{j \in \mathcal{N}_p \cap \mathcal{V}_l} \Bigl(\text{sign}(j-p) \lambda_l^{pj,k} - \rho\, x_l^{(j),k}\Bigr)^\top x_l^{(p)} \\+ \frac{\rho}{2}\sum_{l \in S_p} D_{p,l}\Bigl(x_l^{(p)}\Bigr)^2\,.
	\end{multline}
	However, node~$p$ can solve~\eqref{Eq:PartialProbSolvedAtEachNodeC1} only if it knows~$x_l^{(j),k}$ and~$\lambda_l^{pj,k}$, for $j \in \mathcal{N}_p \cap \mathcal{V}_l$ and $l \in S_p$. This is possible if, in the previous iteration, it received the respective copies of~$x_l$ from its neighbors. This is also enough for knowing~$\lambda_l^{pj,k}$, although we will see later that no node needs	to know each~$\lambda_l^{pj,k}$ individually.	The proof of the following lemma, in Appendix~\ref{App:EquivalenceNodesC1}, shows how we obtained~\eqref{Eq:PartialC1_2} from~\eqref{Eq:PartialC1_1}.
	\begin{Lemma}\label{Lem:EquivalenceNodesC1}
		\eqref{Eq:PartialC1_1} and \eqref{Eq:PartialC1_2} are equivalent.
	\end{Lemma}

	We just saw how~\eqref{Eq:ExtADMMAlg1} yields $|\mathcal{C}_1|$ problems with the format of~\eqref{Eq:PartialProbSolvedAtEachNodeC1} that can be solved in parallel by all the nodes with color~$1$. For the other colors, the analysis is the same with one minor difference: in the second term of~\eqref{Eq:PartialProbSolvedAtEachNodeC1} we have~$x_l^{(j),k+1}$ from the neighbors with a smaller color and~$x_l^{(j),k}$ from the nodes with a larger color.

	\begin{algorithm}
    \caption{Algorithm for a connected variable}
    \algrenewcommand\algorithmicrequire{\textbf{Initialization:}}
    \label{Alg:Conn}
    \begin{algorithmic}[1]
    \small
    \Require for all~$p \in \mathcal{V}$, $l \in S_p$, set $\gamma_{l}^{(p),1} = x_l^{(p),1} = 0$; $k=1$
    \Repeat
    \For{$c =1,\ldots,C$}  \label{SubAlg:Conn_ForColors}
        \ForAll{$p \in \mathcal{C}_c$ [in parallel]}
        \ForAll{$l \in S_p$}
            $$
                v_l^{(p),k} = \gamma_l^{(p),k}-
                \rho \sum_{\begin{subarray}{c}
                             j \in \mathcal{N}_p \cap \mathcal{V}_l \\
                             C(j) < c
                           \end{subarray}
                }x_l^{(j),k+1} - \rho \sum_{\begin{subarray}{c}
                             j \in \mathcal{N}_p \cap \mathcal{V}_l \\
                             C(j) > c
                           \end{subarray}
                }x_l^{(j),k}
            $$
            %\vspace{-0.3cm}
            \label{SubAlg:Conn_Vec}
         \EndFor
        \Statex
        \State Set~$x_{S_p}^{(p),k+1}$ as the solution of
            $$
            \underset{x_{S_p}^{(p)}=\{x_l^{(p)}\}_{l \in S_p}}{\arg\min}
            \,
						f_p(x_{S_p}^{(p)}) + \sum_{l \in S_p} {v_l^{(p),k}}^\top x_l^{(p)}  + \frac{\rho}{2}\sum_{l \in S_p} D_{p,l}\Bigl(x_l^{(p)}\Bigr)^2
            $$
           \label{SubAlg:Conn_Prob}
        \State For each component~$l \in S_p$, send~$x_l^{(p),k+1}$ to $\mathcal{N}_p \cap \mathcal{V}_l$
        \label{SubAlg:Conn_Comm}
    \EndFor
    \EndFor \label{SubAlg:Conn_EndForColors}

    \ForAll{$p \in \mathcal{V}$ and $l \in S_p$ [in parallel]} \vspace{0.15cm}
    \hfill

        $
            \gamma_l^{(p),k+1} = \gamma_l^{(p),k} + \rho \sum_{j \in \mathcal{N}_p \cap \mathcal{V}_l} (x_l^{(p),k+1} -  x_l^{(j),k+1})
        $\label{SubAlg:Conn_DualVar}  \vspace{0.15cm}

    \EndFor
    \State $k \gets k+1$
    \Until{some stopping criterion is met}
    \end{algorithmic}
  \end{algorithm}

  The resulting algorithm is shown in Algorithm~\ref{Alg:Conn}. There is a clear correspondence between the structure of Algorithm~\ref{Alg:Conn} and equations~\eqref{Eq:ExtADMMAlg1}-\eqref{Eq:ExtADMMAlg4}: steps~\ref{SubAlg:Conn_ForColors}-\ref{SubAlg:Conn_EndForColors} correspond to~\eqref{Eq:ExtADMMAlg1}-\eqref{Eq:ExtADMMAlg3}, and the loop in step~\ref{SubAlg:Conn_DualVar} corresponds to~\eqref{Eq:ExtADMMAlg4}. In steps~\ref{SubAlg:Conn_ForColors}-\ref{SubAlg:Conn_EndForColors}, nodes work according to their colors, with the same colored nodes working in parallel. Each node computes the vector~$v$ in step~\ref{SubAlg:Conn_Vec}, solves the optimization problem in step~\ref{SubAlg:Conn_Prob}, and then sends the new estimates of~$x_l$ to the neighbors that also depend on~$x_l$, for~$l \in S_p$. Note the introduction of extra notation in step~\ref{SubAlg:Conn_Vec}: $C(p)$ is the color of node~$p$. The computation of~$v_l^{(p),k}$ in that step requires~$x_l^{(j),k}$ from the neighbors with larger colors and~$x_l^{(j),k+1}$ from the neighbors with smaller colors. While the former is obtained from the previous iteration, the latter is obtained at the current iteration, after the respective nodes execute step~\ref{SubAlg:Conn_Comm}. Regarding the problem in step~\ref{SubAlg:Conn_Prob}, it involves the private function of node~$p$, $f_p$, to which is added a linear and a quadratic term. This fulfills our requirement that all operations involving~$f_p$ be performed at node~$p$.

  Note that the update of the dual variables in step~\ref{SubAlg:Conn_DualVar} is different from~\eqref{Eq:PartialDualVariableUpdate}. In particular, all the $\lambda$'s at node~$p$ were condensed into a single dual variable~$\gamma^{(p)}$. This was done because the optimization problem~\eqref{Eq:PartialProbSolvedAtEachNodeC1} does not depend on the individual $\lambda_l^{pj}$'s, but only on~$\gamma_l^{(p),k} := \sum_{j \in \mathcal{N}_p \cap \mathcal{V}_l} \text{sign}(j-p) \lambda_l^{pj,k}$. If we replace
  \begin{equation}\label{Eq:PartialNewDualUpdate}
  	\lambda_l^{ij,k+1} = \lambda_l^{ij,k} + \rho \,\,\text{sign}(j-i)\bigl(x_l^{(i),k+1} - x_l^{(j),k+1}\bigr)
  \end{equation}
  in the definition of~$\gamma_l^{(p),k}$, we obtain the update of step~\ref{SubAlg:Conn_DualVar}. The extra ``sign'' in~\eqref{Eq:PartialNewDualUpdate} (w.r.t. \eqref{Eq:PartialDualVariableUpdate}) was necessary to take into account the extension of the definition of the dual variable~$\lambda_l^{ij}$ for~$i>j$ (see Appendix~\ref{App:EquivalenceNodesC1}).

\mypar{Convergence}
	Apart from manipulations, Algorithm~\ref{Alg:Conn} results from the application of the Extended ADMM to problem~\eqref{Eq:PartialManip2}. Consequently, the conclusions of Theorem~\ref{Teo:ConvergenceADMM} apply if we prove that~\eqref{Eq:PartialManip2} satisfies the conditions of that theorem.
	\begin{Lemma}\label{Lem:FullColumnRank}
		Each matrix~$\bar{A}^c$ in~\eqref{Eq:PartialManip2} has full column rank.
	\end{Lemma}
	\begin{proof}
		Let~$c$ be any color in~$\{1,2,\ldots,C\}$. By definition, $\bar{A}^c = \text{diag}(\bar{A}_1^c,\bar{A}_2^c,\ldots,\bar{A}_n^c)$; therefore, we have to prove that each~$\bar{A}_l^c$ has full column rank, for~$l=1,2,\ldots,n$. Let then~$c$ and~$l$ be fixed. We are going to prove that $(\bar{A}_l^c)^\top \bar{A}_l^c$, a square matrix, has full rank, and therefore $\bar{A}_l^c$ has full column rank. Since~$\bar{A}_l = \begin{bmatrix}\bar{A}_1^c & \bar{A}_2^c & \cdots & \bar{A}_n^c\end{bmatrix}$, $(\bar{A}_l^c)^\top \bar{A}_l^c$ corresponds to the $l$th block in the diagonal of the matrix~$A_l^\top A_l$, the Laplacian matrix of the induced subgraph~$\mathcal{G}_l$. 	By assumption, in this section all induced subgraphs are connected. This means each node in~$\mathcal{G}_l$ has at least one neighbor also in~$\mathcal{G}_l$ and hence each entry in the diagonal of~$A_l^\top A_l$ is greater than zero.\footnote{We are implicitly excluding the pathological case where a component~$x_l$ appears in only one node, say node~$p$; this would lead to a Laplacian matrix~$A_l^\top A_l$ equal to~$0$. This case is easily addressed by redefining~$f_p$, the function at node~$p$, to~$\tilde{f}_p(\cdot) = \inf_{x_l} f_p(\ldots,x_l,\ldots)$.} The same happens to the entries in the diagonal of~$(\bar{A}_l^c)^\top \bar{A}_l^c$. In fact, these are the only nonzero entries of~$(\bar{A}_l^c)^\top \bar{A}_l^c$, as this is a diagonal matrix. This is because~$(\bar{A}_l^c)^\top \bar{A}_l^c$ corresponds to the Laplacian entries of nodes that have the same color, which are never neighbors. Therefore, $(\bar{A}_l^c)^\top \bar{A}_l^c$ has full rank.
	\end{proof}

	The following corollary, whose proof is omitted, is a straightforward consequence of Theorem~\ref{Teo:ConvergenceADMM} and Lemma~\ref{Lem:FullColumnRank}.
	\begin{Corollary}\label{Cor:ConvergenceConn}
		Let Assumptions~\ref{Ass:PartialVariable}-\ref{Ass:Function} hold and let the variable be connected. Let also one of the following conditions hold:
		\begin{enumerate}
			\item the network is bipartite, i.e., $C=2$, or
			\item each~$\sum_{p \in \mathcal{C}_c} f_p(x_{S_p})$ is strongly convex, $c=1,\ldots,C$.
		\end{enumerate}
		Then, the sequence $\{x_{S_p}^{(p),k}\}_{k=1}^\infty$ at node~$p$, produced by Algorithm~\ref{Alg:Conn}, converges to~$x_{S_p}^\star$, where~$x^\star$ solves~\eqref{Eq:PartialProb}.
	\end{Corollary}
	As stated before, it is believed that the Extended ADMM converges for~$C>2$ even when none of the~$g_c$'s is strongly convex (just closed and convex). However, it is required that each~$E_c$ has full column rank. This translates into the belief that Algorithm~\ref{Alg:Conn} converges for any network, provided each~$f_p$ is closed and convex and each matrix~$\bar{A}^c$ in~\eqref{Eq:PartialManip2} has full column rank. The last condition is the content of Lemma~\ref{Lem:FullColumnRank}.

\mypar{Comparison with other algorithms}
	Algorithm~\ref{Alg:Conn} is a generalization of D-ADMM~\cite{Mota12-DistributedBP}: by violating Assumption~\ref{Ass:PartialVariable} and making $S_p = \{1,\ldots,n\}$ for all~$p$, the variable becomes global and Algorithm~\ref{Alg:Conn} becomes exactly D-ADMM. This is a generalization indeed, for Algorithm~\ref{Alg:Conn} cannot be obtained from D-ADMM. The above fact is not surprising since Algorithm~\ref{Alg:Conn} was derived using the same set of ideas as D-ADMM, but adapted to a partial variable. Each iteration of Algorithm~\ref{Alg:Conn} (resp. D-ADMM) involves communicating $\sum_{p=1}^P |S_p|$ (resp. $nP$) numbers. Under Assumption~\ref{Ass:PartialVariable}, $\sum_{p=1}^P |S_p| < nP$, and thus there is a clear per-iteration gain in solving~\eqref{Eq:PartialProb} with Algorithm~\ref{Alg:Conn}. Although Assumption~\ref{Ass:PartialVariable} can be ignored in the sense that Algorithm~\ref{Alg:Conn} still works without it, we considered that assumption to make clear the type of problems addressed in this paper.

	\begin{algorithm}
    \caption{\cite{Kekatos12-DistributedRobustPowerStateEstimation}}
    \algrenewcommand\algorithmicrequire{\textbf{Initialization:}}
    \label{Alg:Kekatos}
    \begin{algorithmic}[1]
    \small
    \Require for all~$p \in \mathcal{V}$, $l \in S_p$, set $\gamma_{l}^{(p),1} = x_l^{(p),1} = 0$; $k=1$
    \Repeat
      \ForAll{$p \in \mathcal{V}$ [in parallel]}
        \ForAll{$l \in S_p$}
            $$
                v_l^{(p),k} = \gamma_l^{(p),k} - \frac{\rho}{2} \Bigl(D_{p,l} x_l^{(p),k} + \sum_{j \in \mathcal{N}_p \cap \mathcal{V}_l} x_l^{(j),k}\Bigr)
            $$
            \label{SubAlg:2BlockPartialVec}
         \EndFor
        \Statex
        \State Set $x_{S_p}^{(p),k+1}$ as the solution of
            $$
            \underset{x_{S_p}^{(p)}=\{x_l^{(p)}\}_{l \in S_p}}{\arg\min}
            \,
						f_p(x_{S_p}^{(p)}) + \sum_{l \in S_p} {v_l^{(p),k}}^\top x_l^{(p)}  + \frac{\rho}{2}\sum_{l \in S_p} D_{p,l}\Bigl(x_l^{(p)}\Bigr)^2
            $$
           \label{SubAlg:2BlockPartialProb}
        \State For each component~$l \in S_p$, send~$x_l^{(p),k+1}$ to $\mathcal{N}_p \cap \mathcal{V}_l$
        \label{SubAlg:2BlockPartialComm}
    \EndFor

    \ForAll{$p \in \mathcal{V}$ and $l \in S_p$ [in parallel]} \vspace{0.15cm}
    \hfill

        $
            \gamma_l^{(p),k+1} = \gamma_l^{(p),k} + \frac{\rho}{2} \sum_{j \in \mathcal{N}_p \cap \mathcal{V}_l} (x_l^{(p),k+1} -  x_l^{(j),k+1})
        $\label{SubAlg:2BlockPartialDualVarUpdt}  \vspace{0.15cm}

    \EndFor
    \State $k \gets k+1$
    \Until{some stopping criterion is met}
    \end{algorithmic}
  \end{algorithm}

  We mentioned before that the algorithm in~\cite{Kekatos12-DistributedRobustPowerStateEstimation} is the only one we found in the literature that efficiently solves~\eqref{Eq:PartialProb} in the same scenarios as Algorithm~\ref{Alg:Conn}. For comparison purposes, we show it as Algorithm~\ref{Alg:Kekatos}. Algorithms~\ref{Alg:Conn} and~\ref{Alg:Kekatos} are very similar in format, although their derivations are considerably different. In particular, Algorithm~\ref{Alg:Kekatos} is derived from the $2$-block ADMM and thus it has stronger convergence guarantees. Namely, it does not require the network to be bipartite nor any function to be strongly convex (cf. Corollary~\ref{Cor:ConvergenceConn}). Also, it does not require any coloring scheme and, instead, all nodes perform the same tasks in parallel. Note also that the updates of~$v_l^{(p)}$ and~$\gamma_l^{(p)}$ are different in both algorithms. In the same way that Algorithm~\ref{Alg:Conn} was derived using the techniques of D-ADMM, Algorithm~\ref{Alg:Kekatos} was derived using the techniques of~\cite{Zhu09-DistributedInNetworkChannelCoding}. And, as in the experimental results of~\cite{Mota12-DistributedBP,Mota12-DADMM}, we will observe in Section~\ref{Sec:ExperimentalResults} that Algorithm~\ref{Alg:Conn} always requires less communications than Algorithm~\ref{Alg:Kekatos}. Next, we propose a modification to Algorithms~\ref{Alg:Conn} and~\ref{Alg:Kekatos} that makes them applicable to a non-connected variable.

\section{Non-connected Case}
\label{Sec:NonConnected}

	So far, we have assumed a connected variable in~\eqref{Eq:PartialProb}. In this section, the variable will be non-connected, i.e., it will have at least one component that induces a non-connected subgraph. 	In this case, problems~\eqref{Eq:PartialProb} and~\eqref{Eq:PartialManip1} are no longer equivalent and, therefore, the derivations that follow do not apply. We propose a small trick to make these problems equivalent.

	Let~$x_l$ be a component whose induced subgraph~$\mathcal{G}_l = (\mathcal{V}_l,\mathcal{E}_l)$ is non-connected. Then, the constraint~$x_l^{(i)} = x_l^{(j)}$, $(i,j) \in \mathcal{E}_l$, in~\eqref{Eq:PartialManip1} fails to enforce equality on all the copies of~$x_l$. To overcome this, we propose creating a ``virtual'' path to connect the disconnected components of~$\mathcal{G}_l$. This will allow the nodes in~$\mathcal{G}_l$ to reach an agreement on an optimal value for~$x_l$. Since our goal is to minimize communications, we would like to find the ``shortest path'' between these disconnected components, that is, to find an optimal \textit{Steiner tree}.

	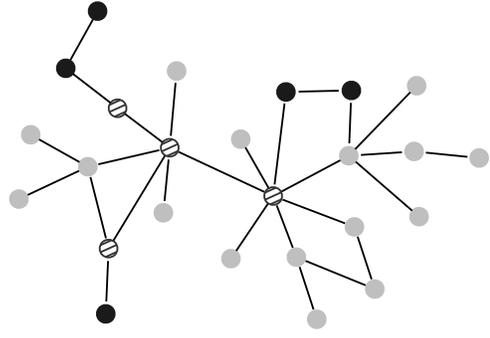
\begin{figure}[t]
  \centering
  \psscalebox{0.87}{
    \begin{pspicture}(6.0,5.3)
        \def\nodesimp{
          \pscircle*[linecolor=black!25!white](0,0){0.15}
        }
        \def\nodeterm{
          \pscircle*[linecolor=black!90!white](0,0){0.15}
        }
         \def\nodestein{
           \pscircle[fillstyle=hlines*,linecolor=black!80!white,hatchcolor=black!80!white,hatchsep=2pt,hatchangle=160](0,0){0.15}
         }

        %\rput(0.4,4.1){\rnode{C1}{\nodesimp}}   \rput(0.4,4.1){\small \textcolor{white}{$1$}}
        %\ncline[nodesep=0.33cm,linewidth=0.9pt]{-}{C1}{C2}
        %\rput[lb](0.0,4.45){ $f_1(x_1,x_2)$}
        \psrotate(2.5,3.0){45}{
       	\rput(1.247200,2.922000){\rnode{N0}{\nodesimp}}	  %\rput(1.247200,2.922000){0}
				\rput(2.019781,3.556917){\rnode{N1}{\nodestein}}	  %\rput(2.019781,3.556917){1}
				\rput(0.929200,4.236700){\rnode{N2}{\nodesimp}}	  %\rput(0.929200,4.236700){2}
				\rput(2.615100,1.917500){\rnode{N3}{\nodestein}}	  %\rput(2.615100,1.917500){3}
				\rput(2.205218,1.005361){\rnode{N4}{\nodesimp}}	  %\rput(2.205218,1.005361){4}
				\rput(2.879054,2.882035){\rnode{N5}{\nodesimp}}	  %\rput(2.879054,2.882035){5}
				\rput(3.867600,1.535600){\rnode{N6}{\nodesimp}}	  %\rput(3.867600,1.535600){6}
				\rput(1.882690,4.547475){\rnode{N7}{\nodestein}}	  %\rput(1.882690,4.547475){7}
				\rput(0.658078,5.199245){\rnode{N8}{\nodesimp}}	  %\rput(0.658078,5.199245){8}
				\rput(4.602102,2.214207){\rnode{N9}{\nodeterm}}	  %\rput(4.602102,2.214207){9}
				\rput(4.619707,0.876559){\rnode{N10}{\nodesimp}}	  %\rput(4.619707,0.876559){10}
				\rput(3.159143,0.705314){\rnode{N11}{\nodesimp}}  %\rput(3.159143,0.705314){11}
				\rput(-0.165185,4.631585){\rnode{N12}{\nodesimp}}	%\rput(-0.165185,4.631585){12}
				\rput(1.755478,5.539351){\rnode{N13}{\nodeterm}}	  %\rput(1.755478,5.539351){13}
				\rput(2.923960,4.312936){\rnode{N14}{\nodesimp}}	  %\rput(2.923960,4.312936){14}
				\rput(5.358214,1.559764){\rnode{N15}{\nodesimp}}	  %\rput(5.358214,1.559764){15}
				\rput(0.268552,3.127543){\rnode{N16}{\nodestein}}	%\rput(0.268552,3.127543){16}
				\rput(3.968349,0.117789){\rnode{N17}{\nodesimp}}	  %\rput(3.968349,0.117789){17}
				\rput(2.717151,5.813546){\rnode{N18}{\nodeterm}}	  %\rput(2.717151,5.813546){18}
				\rput(1.753251,0.113327){\rnode{N19}{\nodesimp}}	  %\rput(1.753251,0.113327){19}
				\rput(-0.468233,2.451415){\rnode{N20}{\nodeterm}}	%\rput(-0.468233,2.451415){20}
				\rput(2.707175,-0.186721){\rnode{N21}{\nodesimp}}	%\rput(2.707175,-0.186721){21}
				\rput(3.878797,2.904735){\rnode{N22}{\nodeterm}}	  %\rput(3.878797,2.904735){22}
				\rput(1.482241,1.696233){\rnode{N23}{\nodesimp}}	  %\rput(1.482241,1.696233){23}
				\rput(5.252491,0.102231){\rnode{N24}{\nodesimp}}	  %\rput(5.252491,0.102231){24}
				}

				\ncline[nodesep=0.140000cm]{-}{N16}{N2}
				\ncline[nodesep=0.140000cm]{-}{N4}{N21}
				\ncline[nodesep=0.140000cm]{-}{N9}{N22}
				\ncline[nodesep=0.140000cm]{-}{N0}{N1}
				\ncline[nodesep=0.140000cm]{-}{N1}{N2}
				\ncline[nodesep=0.140000cm]{-}{N1}{N3}
				\ncline[nodesep=0.140000cm]{-}{N1}{N7}
				\ncline[nodesep=0.140000cm]{-}{N7}{N13}
				\ncline[nodesep=0.140000cm]{-}{N1}{N14}
				\ncline[nodesep=0.140000cm]{-}{N1}{N16}
				\ncline[nodesep=0.140000cm]{-}{N2}{N8}
				\ncline[nodesep=0.140000cm]{-}{N2}{N12}
				\ncline[nodesep=0.140000cm]{-}{N3}{N4}
				\ncline[nodesep=0.140000cm]{-}{N3}{N5}
				\ncline[nodesep=0.140000cm]{-}{N3}{N6}
				\ncline[nodesep=0.140000cm]{-}{N3}{N11}
				\ncline[nodesep=0.140000cm]{-}{N3}{N22}
				\ncline[nodesep=0.140000cm]{-}{N3}{N23}
				\ncline[nodesep=0.140000cm]{-}{N4}{N19}
				\ncline[nodesep=0.140000cm]{-}{N6}{N9}
				\ncline[nodesep=0.140000cm]{-}{N6}{N10}
				\ncline[nodesep=0.140000cm]{-}{N6}{N15}
				\ncline[nodesep=0.140000cm]{-}{N6}{N17}
				\ncline[nodesep=0.140000cm]{-}{N10}{N24}
				\ncline[nodesep=0.140000cm]{-}{N11}{N21}
				\ncline[nodesep=0.140000cm]{-}{N13}{N18}
				\ncline[nodesep=0.140000cm]{-}{N16}{N20}

        %\psgrid
    \end{pspicture}
  }
  \vspace{-0.2cm}
  \caption{
			Example of an optimal Steiner tree: black nodes are required and striped nodes are Steiner.
  }
  \label{Fig:SteinerTree}
  \end{figure}

\mypar{Steiner tree problem}
	Let~$\mathcal{G} = (\mathcal{V},\mathcal{E})$ be an undirected graph and let~$\mathcal{R} \subseteq \mathcal{V}$ be a set of \textit{required nodes}. A Steiner tree is any tree in $\mathcal{G}$ that contains the required nodes, i.e., it is an acyclic connected graph $(\mathcal{T},\mathcal{F}) \subseteq \mathcal{G}$ such that $\mathcal{R} \subseteq \mathcal{T}$. The set of nodes in the tree that are not required are called \textit{Steiner nodes}, and will be denoted with $\mathcal{S}:= \mathcal{T} \backslash \mathcal{R}$. In the \textit{Steiner tree problem}, each edge $(i,j) \in \mathcal{E}$ has a cost~$c_{ij}$ associated, and the goal is to find a Steiner tree whose edges have a minimal cost. This is exactly our problem if we make~$c_{ij} = 1$ for all edges and~$\mathcal{R} = \mathcal{V}_l$. The Steiner tree problem is illustrated in Fig.~\ref{Fig:SteinerTree}, where the required nodes are black and the Steiner nodes are striped. Unfortunately, computing optimal Steiner trees is NP-hard~\cite{Garey77-SteinerNPComplete}. There are, however, many heuristic algorithms, some even with approximation guarantees. The Steiner tree problem can be formulated as~\cite{Williamson02-PrimalDualMethodApproximation}
	\begin{equation}\label{Eq:SteinerProb}
		\begin{array}{ll}
			\underset{\{z_{ij}\}_{(i,j) \in \mathcal{E}}}{\text{minimize}} & \underset{(i,j)\in \mathcal{E}}{\sum} c_{ij} z_{ij} \vspace{0.1cm}\\
			\text{subject to} & \underset{\begin{subarray}{c}i \in \mathcal{U} \\ j \not\in \mathcal{U} \end{subarray}}{\sum} z_{ij} \geq 1\,,\quad \forall_\mathcal{U} \,:\, 0<|\mathcal{U} \cap \mathcal{R}| <|\mathcal{R}| \\
			& z_{ij} \in \{0,1\}\,,\quad (i,j) \in \mathcal{E}\,,
		\end{array}
	\end{equation}
	where~$\mathcal{U}$ in the first constraint is any subset of nodes that separates at least two required nodes. The optimization variable is constrained to be binary, and an optimal value $z_{ij}^\star = 1$ means that edge~$(i,j)$ was selected for the Steiner tree. Let~$h(z):= \sum_{(i,j)\in \mathcal{E}} c_{ij} z_{ij}$ denote the objective of~\eqref{Eq:SteinerProb}. We say that an algorithm for~\eqref{Eq:SteinerProb} has an approximation ratio of~$\alpha$ if it produces a feasible point~$\bar{z}$ such $h(\bar{z}) \leq \alpha h(z^\star)$, for any problem instance. For example, the primal-dual algorithm for combinatorial problems~\cite{Williamson02-PrimalDualMethodApproximation,Goemans97-PrimalDualMethod} has an approximation ratio of~$2$. This number has been decreased in a series of works, the smallest one being $1+\text{ln}\,	3/2\simeq 1.55$, provided by~\cite{Robins00-ImprovedSteinerTree}.

\mypar{Algorithm generalization}
	To make Algorithms~\ref{Alg:Conn} and~\ref{Alg:Kekatos} applicable to a non-connected variable, we propose the following preprocessing step. For every component~$x_l$ that induces a disconnected subgraph~$\mathcal{G}_l=(\mathcal{V}_l,\mathcal{E}_l)$, compute a Steiner tree $(\mathcal{T}_l,\mathcal{F}_l) \subseteq \mathcal{G}$ using~$\mathcal{V}_l$ as required nodes. Let~$\mathcal{S}_l := \mathcal{T}_l \backslash \mathcal{V}_l$ denote the Steiner nodes in that tree. The functions of these Steiner nodes do not depend on~$x_l$, i.e., $x_l \not\in S_p$ for all~$p \in \mathcal{S}_l$. Define a new induced graph as $\mathcal{G}_l'=(\mathcal{V}_l',\mathcal{E}_l')$, with $\mathcal{V}_l' := \mathcal{T}_l$ and~$\mathcal{E}_l' := \mathcal{E}_l\cup\mathcal{F}_l$. Then, we can create copies of~$x_l$ in all nodes in~$\mathcal{V}_l'$, and write~\eqref{Eq:PartialProb} equivalently as
	\begin{equation}\label{Eq:NonConnectedManip}
	  \begin{array}{ll}
		  \underset{\{\bar{x}_l\}_{l=1}^n}{\text{minimize}} & f_1(x_{S_1}^{(1)}) + f_2(x_{S_2}^{(2)}) + \cdots + f_P(x_{S_P}^{(P)}) \\
		  \text{subject to} & x_l^{(i)} = x_l^{(j)},\quad (i,j) \in \mathcal{E}_l'\,,\,\,l=1,\ldots,n\,,
		\end{array}
	\end{equation}
	where~$\bar{x}_l := \{x_l^{(p)}\}_{p \in \mathcal{V}_l'}$ denotes the set of all copies of~$x_l$, and~$\{\bar{x}_l\}_{l=1}^L$, the optimization variable, represents the set of all copies. Note that the function at node~$p$ remains unchanged: it only depends on $x_{S_p}^{(p)} := \{x_l^{(p)}\}_{l \in S_p}$, although node~$p$ can now have more copies, namely, $x_{S_p \cup S_p'}^{(p)}$, where $S_p'$ is the set of components of which node~$p$ is a Steiner node. Of course, when a component~$x_l$ is connected, we set $\mathcal{G}_l' = \mathcal{G}_l$; also, if a node~$p$ is not Steiner for any component, $S_p' = \emptyset$. If we repeat the analysis of the previous section replacing problem~\eqref{Eq:PartialManip1} by~\eqref{Eq:NonConnectedManip}, we get Algorithm~\ref{Alg:NonConn}.

  \begin{algorithm}
    \caption{Algorithm for a non-connected variable}
    \label{Alg:NonConn}
    \begin{algorithmic}[1]
    \small
    \algrenewcommand\algorithmicrequire{\textbf{Preprocessing (centralized):}}
    \Require
    \State Set $S_p' = \emptyset$ for all~$p \in \mathcal{V}$, and $\mathcal{V}_l'=\mathcal{V}_l$ for all~$l = \{1,\ldots,n\}$
    \ForAll{$l \in \{1,\ldots,n\}$ such that $x_l$ is non-connected}
			\State Compute a Steiner tree $(\mathcal{T}_l, \mathcal{F}_l)$, where $\mathcal{V}_l$ are required nodes
			\State Set $\mathcal{V}_l'= \mathcal{T}_l$ and $\mathcal{S}_l := \mathcal{T}_l \backslash \mathcal{V}_l$ (Steiner nodes)
			\State For all $p \in \mathcal{S}_l$, $S_p' = S_p' \cup \{x_l\}$
    \EndFor
    \Statex
    \algrenewcommand\algorithmicrequire{\textbf{Main algorithm (distributed):}}
    \Require
    \algrenewcommand\algorithmicrequire{\textbf{Initialization:}}
    \Require Set $\gamma_{l}^{(p),1}\!\! = x_l^{(p),1}\!\! = 0$, for $l \in S_p \cup S_p'$, $p \in \mathcal{V}$; $k=1$
    \Repeat
		\label{SubAlg:NonConn_FirstStep}
    \For{$c =1,\ldots,C$}  \label{SubAlg:NonConn_ForColors}
        \ForAll{$p \in \mathcal{C}_c$ [in parallel]}
        \ForAll{$l \in S_p \cup S_p'$}
            $$
                v_l^{(p),k} = \gamma_l^{(p),k}-
                \rho \sum_{\begin{subarray}{c}
                             j \in \mathcal{N}_p \cap \mathcal{V}_l' \\
                             C(j) < c
                           \end{subarray}
                }x_l^{(j),k+1} - \rho \sum_{\begin{subarray}{c}
                             j \in \mathcal{N}_p \cap \mathcal{V}_l' \\
                             C(j) > c
                           \end{subarray}
                }x_l^{(j),k}
            $$
            %\vspace{-0.3cm}
            \label{SubAlg:NonConn_Vec}
         \EndFor
        \Statex
        \State Set~$x_{S_p \cup S_p'}^{(p),k+1}$ as the solution of
            $$
            \underset{x_{S_p\cup S_p'}^{(p)}}{\arg\min}
            \,
						f_p(x_{S_p}^{(p)}) + \sum_{l \in S_p\cup S_p'} \biggl( {v_l^{(p),k}}^\top x_l^{(p)}  + \frac{\rho}{2} D_{p,l}\Bigl(x_l^{(p)}\Bigr)^2 \biggr)
            $$
           \label{SubAlg:NonConn_Prob}
        \State For each~$l \in S_p\cup S_p'$, send~$x_l^{(p),k+1}$ to $\mathcal{N}_p \cap \mathcal{V}_l'$
        \label{SubAlg:NonConn_Comm}
    \EndFor
    \EndFor \label{SubAlg:NonConn_EndForColors}

    \ForAll{$p \in \mathcal{V}$ and $l \in S_p \cup S_p'$ [in parallel]} \vspace{0.15cm}
    \hfill

        $
            \gamma_l^{(p),k+1} = \gamma_l^{(p),k} + \rho \sum_{j \in \mathcal{N}_p \cap \mathcal{V}_l'} (x_l^{(p),k+1} -  x_l^{(j),k+1})
        $\label{SubAlg:NonConn_DualVar}  \vspace{0.15cm}

    \EndFor
    \State $k \gets k+1$
    \Until{some stopping criterion is met}
    \label{SubAlg:NonConn_FinalStep}
    \end{algorithmic}
  \end{algorithm}

  Algorithm~\ref{Alg:NonConn} has two parts: a preprocessing step, which is new, and the main algorithm, which it essentially Algorithm~\ref{Alg:Conn} with some small adaptations. We assume the preprocessing step can be done in a centralized way, before the execution of the main algorithm. In fact, the preprocessing only requires knowing the communication network~$\mathcal{G}$ and the nodes' dependencies, but not the specific the functions~$f_p$. Regarding the main algorithm, it is similar to Algorithm~\ref{Alg:Conn} except that each node, in addition to estimating the components its function depends on, it also estimates the components for which it is a Steiner node. The additional computations are, however, very simple: if node~$p$ is a Steiner node for component~$x_l$, it updates it as  $x_l^{(p),k+1} = -(1/(\rho \,D_{p,l}))v_l^{(p),k}$ in step~\ref{SubAlg:NonConn_Prob}; since~$f_p$ does not depend on~$x_l$, the problem corresponding to the update of~$x_l$ becomes a quadratic problem for which there is a closed-form solution. Note that~$D_{p,l}$ is now defined as the degree of node~$p$ in the subgraph~$\mathcal{G}_l'$. The steps we took to generalize Algorithm~\ref{Alg:Conn} to a non-connected variable can be easily applied the same way to Algorithm~\ref{Alg:Kekatos}.

\section{Applications}
\label{Sec:Applications}

	In this section we describe how the proposed algorithms can be used to solve distributed MPC and network flow problems.

	\mypar{Distributed MPC} MPC is a popular control strategy for discrete-time systems~\cite{Morari99-MPCPastPresentFuture}. It assumes a state-space model for the system, where the state at time~$t$, here denoted with~$x[t] \in \mathbb{R}^n$, evolves according to $x[t+1] = \Theta^t(x[t],u[t])$, where~$u[t] \in \mathbb{R}^m$ is the input at time~$t$ and~$\Theta^t:\mathbb{R}^n\times \mathbb{R}^m \xrightarrow{} \mathbb{R}^n$ is a map that gives the system dynamics at each time instant~$t$. Given a time-horizon~$T$, an MPC implementation consists of measuring the state at time~$t=0$, computing the desired states and inputs for the next~$T$ time steps, applying~$u[0]$ to the system, setting~$t=0$, and repeating the process. The second step, i.e., computing the desired states and inputs for a given time horizon~$T$, is typically addressed by solving
	\begin{equation}\label{Eq:MPC}
    \begin{array}{ll}
      \underset{\bar{x},\bar{u}}{\text{minimize}} & \Phi(x[T]) + \sum_{t=0}^{T-1} \Psi^t(x[t], u[t]) \\
      \text{subject to} & x[t+1] = \Theta^t(x[t], u[t])\,,\quad t=0,\ldots,T-1 \\
                        & x[0] = x^0\,,
    \end{array}
  \end{equation}
	where the variable is~$(\bar{x},\bar{u}):= (\{x[t]\}_{t=0}^T, \{u[t]\}_{t=0}^{T-1})$. While~$\Phi$ penalizes deviations of the final state~$x[T]$ from our goal, $\Psi^t$ usually measures, for each~$t=0,\ldots,T-1$, some type of energy consumption that we want to minimize. Regarding the constraints of~\eqref{Eq:MPC}, the first one enforces the state to follow the system dynamics, and the second one encodes the initial measurement~$x^0$.

	We solve~\eqref{Eq:MPC} in the following distributed scenario. There is a set of~$P$ systems that communicate through a communication network~$\mathcal{G} = (\mathcal{V}, \mathcal{E})$. Each system has a state~$x_p[t] \in \mathbb{R}^{n_p}$ and a local input~$u_p[t] \in \mathbb{R}^{m_p}$, where~$n_1 + \cdots + n_P = n$ and $m_1 + \cdots + m_P = m$. The state of system~$p$ evolves as $x_p[t+1] = \Theta_p^t\bigl(\{x_j[t],u_j[t]\}_{j \in \Omega_p}\bigr)$, where $\Omega_p \subseteq \mathcal{V}$ is the set of nodes whose state and/or input influences~$x_p$ (we assume $\{p\}\subseteq \Omega_p$ for all~$p$). Note that, in contrast with what is usually assumed, $\Omega_p$ is not necessarily a subset of the neighbors of node~$p$. In other words, two systems that influence each other may be unable to communicate directly. This is illustrated in Fig.~\ref{SubFig:MPCNonConnected} where, for example, the state/input of node~$3$ influences the state evolution of node~$1$ (dotted arrow), but there is no communication link (solid line) between them. Finally, we assume functions~$\Phi$ and~$\Psi^t$ in~\eqref{Eq:MPC} can be decomposed, respectively, as $	\Phi(x[T]) = \sum_{p=1}^P \Phi_p(\{x_j[T]\}_{j \in \Omega_p})$ and $	\Psi^t(x[t],u[t]) = \sum_{p=1}^P \Psi_p^t (\{x_j[t], u_j[t]\}_{j \in \Omega_p})$, where~$\Phi_p$ and~$\Psi_p^t$ are both associated to node~$p$. In sum, we solve
	\begin{equation}\label{Eq:DistributedMPC}
	  \begin{array}{cl}
	  	\underset{\bar{x},\bar{u}}{\min} &
				\sum_{p=1}^P\Bigl[\Phi_p(\{x_j[T]\}_{j \in \Omega_p})
				\\
				&\phantom{aaaaaaaaaaaaaaa}+ \sum_{t=0}^{T-1} \Psi_p^t (\{x_j[t], u_j[t]\}_{j \in \Omega_p})\Bigr] \\
			\text{s.t.} & x_p[t+1] = \Theta_p^t\bigl(\{x_j[t],u_j[t]\}_{j \in \Omega_p}\bigr)\,,\,\, t = 0,\ldots,T-1 \\
			            & x_p[0] = x_p^0 \\
			            & p = 1,\ldots,P\,,
	  \end{array}
	\end{equation}
	where~$x_p^0$ is the initial measurement at node~$p$. The variable in~\eqref{Eq:DistributedMPC} is $(\bar{x},\bar{u}) := \bigl(\{\bar{x}_p\}_{p=1}^{P} , \{\bar{u}_p\}_{p=1}^{P}\bigr)$, where $\bar{x}_p := \{x_p[t]\}_{t=0}^T$ and $\bar{u}_p := \{u_p[t]\}_{t=0}^{T-1}$. Problem~\eqref{Eq:DistributedMPC} can be written as~\eqref{Eq:PartialProb} by making
	\begin{multline*}
		f_p(\{\bar{x}_j,\bar{u}_j\}_{j \in \Omega_p}) = \Phi_p(\{x_j[T]\}_{j \in \Omega_p}) + \text{i}_{x_p[0] = x_p^0}(\bar{x}_p)
									 \\
									 + \sum_{t=0}^{T-1} \Bigl(
																				\Psi_p^t(\{x_j[t],u_j[t]\}_{j \in \Omega_p})
																				+
																				\text{i}_{\Gamma_p^t}(\{\bar{x}_j,\bar{u}_j\}_{j \in \Omega_p})
																			\Bigr)\,,
	\end{multline*}
	where $\text{i}_S(\cdot)$ is the indicator function of the set~$S$, i.e., $\text{i}_S(x) = +\infty$ if $x \not\in S$ and $\text{i}_S(x) = 0$ if $x \in S$, and $\Gamma_p^t := \{\{\bar{x}_j,\bar{u}_j\}_{j \in \Omega_p} \,:\, x_p[t+1] = \Theta_p^t\bigl(\{x_j[t],u_j[t]\}_{j \in \Omega_p}\bigr)\}$.

  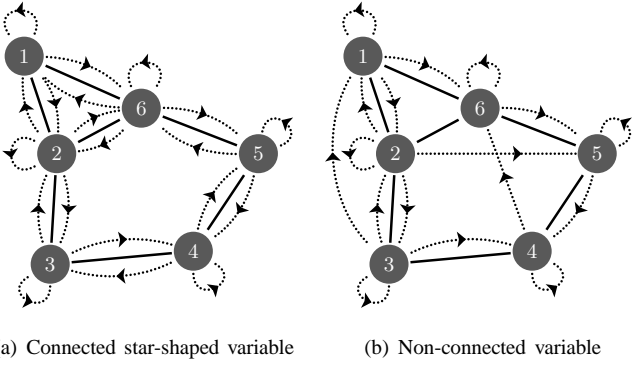
\begin{figure}[t]
  \centering
  \subfigure[Connected star-shaped variable]{\label{SubFig:MPCStar}
    \psscalebox{0.865}{
      \begin{pspicture}(4.5,4.5)
        \def\nodesimp{
          \pscircle*[linecolor=black!65!white](0,0){0.3}
        }
        \def\nodeshigh{
          \pscircle*[linecolor=black!25!white](0,0){0.3}
        }

        \rput(0.4,4.1){\rnode{C1}{\nodesimp}}   \rput(0.4,4.1){\small \textcolor{white}{$1$}}
        \rput(0.9,2.6){\rnode{C2}{\nodesimp}}   \rput(0.9,2.6){\small \textcolor{white}{$2$}}
        \rput(0.8,0.9){\rnode{C3}{\nodesimp}}   \rput(0.8,0.9){\small \textcolor{white}{$3$}}
        \rput(3.0,1.1){\rnode{C4}{\nodesimp}}   \rput(3.0,1.1){\small \textcolor{white}{$4$}}
        \rput(4.0,2.6){\rnode{C5}{\nodesimp}}   \rput(4.0,2.6){\small \textcolor{white}{$5$}}
        \rput(2.2,3.3){\rnode{C6}{\nodesimp}}   \rput(2.2,3.3){\small \textcolor{white}{$6$}}

        \psset{nodesep=0.33cm,linewidth=1.1pt}
        \ncline{-}{C1}{C2}
        \ncline{-}{C1}{C6}
        \ncline{-}{C2}{C3}
        \ncline{-}{C2}{C6}
        \ncline{-}{C3}{C4}
        \ncline{-}{C4}{C5}
        \ncline{-}{C5}{C6}

        \psset{linestyle=dotted,dotsep=0.7pt,arrowsize=7pt,arrowinset=0.2,arrowlength=0.55,linewidth=1.1pt,arcangleA=22,arcangleB=22,ArrowInside=->,ArrowInsidePos=0.56}
        \ncarc{-}{C1}{C2}
        \ncarc{-}{C2}{C1}
        \ncarc{-}{C1}{C6}
        \ncarc{-}{C6}{C1}
        \ncarc{-}{C2}{C3}
        \ncarc{-}{C3}{C2}
        \ncarc{-}{C2}{C6}
        \ncarc{-}{C6}{C2}
        \ncarc{-}{C3}{C4}
        \ncarc{-}{C4}{C3}
        \ncarc{-}{C4}{C5}
        \ncarc{-}{C5}{C4}
        \ncarc{-}{C6}{C5}
        \ncarc{-}{C5}{C6}

				\psset{ncurv=3}
        \nccurve[angleA=65,angleB=125]{-}{C1}{C1}
        \nccurve[angleA=150,angleB=210]{-}{C2}{C2}
        \nccurve[angleA=210,angleB=270]{-}{C3}{C3}
        \nccurve[angleA=270,angleB=330]{-}{C4}{C4}
        \nccurve[angleA=20,angleB=80]{-}{C5}{C5}
        \nccurve[angleA=55,angleB=115]{-}{C6}{C6}

        %\psgrid
      \end{pspicture}
    }
  }
  \hfill
  \subfigure[Non-connected variable]{\label{SubFig:MPCNonConnected}
    \psscalebox{0.865}{
      \begin{pspicture}(4.5,4.5)
        \def\nodesimp{
          \pscircle*[linecolor=black!65!white](0,0){0.3}
        }

        \rput(0.4,4.1){\rnode{C1}{\nodesimp}}   \rput(0.4,4.1){\small \textcolor{white}{$1$}}
        \rput(0.9,2.6){\rnode{C2}{\nodesimp}}   \rput(0.9,2.6){\small \textcolor{white}{$2$}}
        \rput(0.8,0.9){\rnode{C3}{\nodesimp}}   \rput(0.8,0.9){\small \textcolor{white}{$3$}}
        \rput(3.0,1.1){\rnode{C4}{\nodesimp}}   \rput(3.0,1.1){\small \textcolor{white}{$4$}}
        \rput(4.0,2.6){\rnode{C5}{\nodesimp}}   \rput(4.0,2.6){\small \textcolor{white}{$5$}}
        \rput(2.2,3.3){\rnode{C6}{\nodesimp}}   \rput(2.2,3.3){\small \textcolor{white}{$6$}}

        \psset{nodesep=0.33cm,linewidth=1.1pt}
        \ncline{-}{C1}{C2}
        \ncline{-}{C1}{C6}
        \ncline{-}{C2}{C3}
        \ncline{-}{C2}{C6}
        \ncline{-}{C3}{C4}
        \ncline{-}{C4}{C5}
        \ncline{-}{C5}{C6}

        \psset{linestyle=dotted,dotsep=0.7pt,arrowsize=7pt,arrowinset=0.2,arrowlength=0.55,linewidth=1.1pt,arcangleA=22,arcangleB=22,ArrowInside=->,ArrowInsidePos=0.56}
        \ncarc{-}{C1}{C2}
        \ncarc{-}{C2}{C1}
        \ncarc[arcangleA=38,arcangleB=38]{-}{C3}{C1}
        \ncarc{-}{C2}{C3}
        %\ncarc[arcangleA=0,arcangleB=0]{-}{C2}{C4}
        \ncarc[arcangleA=0,arcangleB=0,ArrowInsidePos=0.65]{-}{C2}{C5}
        \ncarc{-}{C3}{C2}
        \ncarc{-}{C3}{C4}
        \ncarc[arcangleA=0,arcangleB=0]{-}{C4}{C6}
        \ncarc{-}{C1}{C6}
        \ncarc{-}{C6}{C5}
        \ncarc{-}{C5}{C4}

        \psset{ncurv=3}
        \nccurve[angleA=65,angleB=125]{-}{C1}{C1}
        \nccurve[angleA=150,angleB=210]{-}{C2}{C2}
        \nccurve[angleA=210,angleB=270]{-}{C3}{C3}
        \nccurve[angleA=270,angleB=330]{-}{C4}{C4}
        \nccurve[angleA=20,angleB=80]{-}{C5}{C5}
        \nccurve[angleA=55,angleB=115]{-}{C6}{C6}

        %\psgrid
      \end{pspicture}
    }
  }

  \caption{
		Two MPC scenarios. Solid lines represent links in the communication network and dotted arrows represent system interactions. \text{(a)} Connected variable where each induced subgraph is a star. \text{(b)} Non-connected variable because node~$5$ is influenced by~$(\bar{x}_2,\bar{u}_2)$, but not none of its neighbors are.
  }
  \label{Fig:MPC}
  \end{figure}

	We illustrate in Fig.~\ref{SubFig:MPCStar} the case where $\Omega_p \subseteq \mathcal{N}_p \cup \{p\}$, i.e.,  the state of node~$p$ is influenced by its own state/input and by the states/inputs of the systems with which it can communicate. Using our terminology, this corresponds to a connected variable, where each induced subgraph is a star: the center of the star is node~$p$, whose state is~$x_p$. Particular cases of this model have been considered, for example, in~\cite{Camponogara02-DistributedMPC,Keviczky06DecentralizedRecedingHorizonControl,Morosan10-DistributedMPCTemperature}, whose solutions are heuristics, and in~\cite{Wakasa08-DecentralizedMPCDualDecomp,Camponogara11-DistributedMPC,Conte12-ComputationalAspectsDistributedMPC,Summers12-DistributedMPConsensus}, whose solutions are optimization-based. The model we propose here is significantly more general, since it can handle scenarios where interacting nodes do not necessarily need to communicate, or even scenarios with a non-connected variable. Both cases are shown in Fig.~\ref{SubFig:MPCNonConnected}. For example, the subgraph induced by $(\bar{x}_3,\bar{u}_3)$ consists of the nodes~$\{1,2,3,4\}$ and is connected. (The reference for connectivity is always the communication network which, in the plots, is represented by solid lines.) Nodes~$1$ and~$3$, however, cannot communicate directly. This is an example of an induced subgraph that is not a star. On the other hand, the subgraph induced by~$(\bar{x}_2,\bar{u}_2)$ consists of the nodes~$\{1,2,3,5\}$. This subgraph is not connected, which implies that the optimization variable is non-connected. Situations like the above can be useful in scenarios where communications links are expensive or hard to establish. For instance, MPC can be used for temperature regulation of buildings~\cite{Morosan10-DistributedMPCTemperature}, where making wired connections between rooms, here viewed as systems, can be expensive. In that case, two adjacent rooms whose temperatures influence each other may not be able to communicate directly. The proposed MPC model can handle this scenario easily.

\mypar{MPC model for the experiments}
	We now present a simple linear MPC model, which will be used in our experiments in Section~\ref{Sec:ExperimentalResults}. Although simple, this model will illustrate all the cases considered above. We assume that systems are coupled though their inputs, i.e., $x_p[t+1] = A_p x_p[t] + \sum_{j \in \Omega_p} B_{pj} u_j[t]$, where~$A_p \in \mathbb{R}^{n_p \times n_p}$ and each~$B_{pj} \in \mathbb{R}^{n_p \times m_j}$ are arbitrary matrices, known only at node~$p$. Also, we assume~$\Phi_p$ and~$\Psi_p^t$ in~\eqref{Eq:DistributedMPC} are, respectively, $\Phi_p(\{x_j[T]\}_{j \in \Omega_p}) = x_p[T]^\top \bar{Q}_p^f x_p[T]$ and $\Psi_p^t(\{x_j[t]\}_{j \in \Omega_p}) = x_p[t]^\top \bar{Q}_p x_p[t] + u_p[t]^\top \bar{R}_p$, where~$\bar{Q}_p$ and~$\bar{Q}_p^f$ are positive semidefinite matrices, and~$\bar{R}_p$ is positive definite. Problem \eqref{Eq:DistributedMPC} then becomes
	\begin{equation}\label{Eq:DistributedMPC_SimpleModel}
  	\begin{array}{cl}
  		\underset{\begin{subarray}{c}x_1,\ldots,x_P\\u_1,\ldots,u_P\end{subarray}}{\text{minimize}} &
  		\sum_{p=1}^P u_p^\top R_p u_p + x_p^\top Q_p x_p \\
  		\text{subject to} & x_p = C_p \{u_j\}_{j \in \mathcal{S}_p} + D_p^0 \,,\,\, p = 1,\ldots,P\,,
  	\end{array}
  \end{equation}
	where, $x_p = (x_p[0],\ldots,x_p[T])$, $u_p = (u_p[0],\ldots,u_p[T-1])$, for each~$p$, and
	\begin{align*}
	  Q_p &= \begin{bmatrix}
		      	I_{T} \otimes \bar{Q}_p & 0\\
		      	0 & \bar{Q}_p^f
		      \end{bmatrix}\,,
		&
		R_p &= I_{T} \otimes \bar{R}_p\,,
		\\
		C_p &=
		\begin{bmatrix}
			0                  & 0                 & \cdots & 0     \\
      B_{p}              & 0                 & \cdots & 0     \\
      A_{pp}B_{p}        & B_{p}             & \cdots & 0     \\
       \vdots            & \vdots            & \ddots & \vdots\\
      A_{pp}^{T-1}B_{p}  & A_{pp}^{T-2}B_{p} & \cdots & B_{p} \\
		\end{bmatrix}\,,
		&D_p^0
		&=
		\begin{bmatrix}
      I           \\
      A_{pp}      \\
      A_{pp}^2    \\
      \vdots      \\
      A_{pp}^{T}  \\
    \end{bmatrix}
    x_p^0\,.
  \end{align*}
	In the entries of matrix~$C_p$, $B_p$ is the horizontal concatenation of the matrices $B_{pj}$, for all~$j \in \Omega_p$. One of the advantages of the model we are using is that all the variables~$x_p$ can be eliminated from~\eqref{Eq:DistributedMPC_SimpleModel}, yielding
	\begin{equation}\label{Eq:DistributedMPC_SimpleModelFinal}
  	\underset{u_1,\ldots,u_P}{\text{minimize}} \,\,\, \sum_{p=1}^P \{u_j\}_{j \in S_p}^\top E_p \{u_j\}_{j \in S_p} + w_p^\top \{u_j\}_{j \in S_p}\,,
  \end{equation}
  where each~$E_p$ is obtained by summing~$R_p$ with $C_p^\top Q_p C_p$ in the correct entries, and $w_p = 2C_p^\top Q_p D_p^0$. Our model thus leads to a very simple problem. In a centralized scenario, where all matrices~$E_p$ and all vectors~$w_p$ are known in the same location, the solution of~\eqref{Eq:DistributedMPC_SimpleModelFinal} can be computed by solving a linear system. Likewise, the problem in step~\ref{SubAlg:Conn_Prob} of Algorithm~\ref{Alg:Conn} (and steps~\ref{SubAlg:2BlockPartialProb} and~\ref{SubAlg:NonConn_Prob} of Algorithms~\ref{Alg:Kekatos} and~\ref{Alg:NonConn}, respectively) boils down to solving a linear system.

  \begin{figure}[h]
  \centering
  \psscalebox{0.865}{
		\begin{pspicture}(7,4.3)
			\def\nodesimp{
          \pscircle*[linecolor=black!65!white](0,0){0.3}
      }

      \rput(1,3.7){\rnode{N1}{\nodesimp}}
      \rput(0.3,2.3){\rnode{N2}{\nodesimp}}
      \rput(0.3,0.3){\rnode{N3}{\nodesimp}}
      \rput(2.7,1.6){\rnode{N4}{\nodesimp}}
      \rput(5,1.3){\rnode{N5}{\nodesimp}}
      \rput(4.0,3.3){\rnode{N6}{\nodesimp}}
      \rput(6.7,2.4){\rnode{N7}{\nodesimp}}

      \rput(1,3.7){\small \textcolor{white}{$1$}}    %\rput[lb](1.282843,3.982843){$f_1$}
      \rput(0.3,2.3){\small \textcolor{white}{$2$}}    %\rput[rb](0.017157,2.282843){$f_2$}
      \rput(0.3,0.3){\small \textcolor{white}{$3$}}  %\rput[rt](0.017157,0.017157){$f_3$}
      \rput(2.7,1.6){\small \textcolor{white}{$4$}}  %\rput[lt](2.782843,1.417157){$f_4$}
      \rput(5,1.3){\small \textcolor{white}{$5$}}      %\rput[lt](5.282843,1.717157){$f_5$}
      \rput(4.0,3.3){\small \textcolor{white}{$6$}}    %\rput[lb](4.082843,3.282843){$f_6$}
      \rput(6.7,2.4){\small \textcolor{white}{$7$}}  %\rput[lb](6.982843,3.482843){$f_7$}

      \psset{nodesep=0.35cm,linestyle=solid,arrowsize=5pt,arrowinset=0.1,labelsep=0.07}
      \ncline[]{->}{N1}{N2}\nbput{$\phi_{12}(x_{12})$}
      \ncline[]{->}{N2}{N3}\nbput{$\phi_{23}(x_{23})$}
      \ncline[]{->}{N2}{N4}\naput{$\phi_{24}(x_{24})$}
      \ncline[]{->}{N4}{N5}\nbput{$\phi_{45}(x_{45})$}
      \ncline[]{->}{N4}{N6}\nbput{$\phi_{46}(x_{46})$}
      %\ncline[nodesep=0.35cm]{-}{N5}{N6}\Bput{$\phi_{56}(x_{56})$}
      \ncline[]{->}{N5}{N7}\nbput{$\phi_{57}(x_{57})$}
      \ncline[]{<-}{N3}{N4}\nbput{$\phi_{43}(x_{43})$}
      \ncline[]{->}{N1}{N6}\naput{$\phi_{16}(x_{16})$}
      \ncline[]{->}{N6}{N7}\naput{$\phi_{67}(x_{67})$}

      %\psgrid
    \end{pspicture}
  }
  \caption{
		A network flow problem: each edge has a variable~$x_{ij}$ representing the flow from node~$i$ to node~$j$ and also has a cost function~$\phi_{ij}(x_{ij})$.
	}
  \label{Fig:NetworkFlow}
  \end{figure}
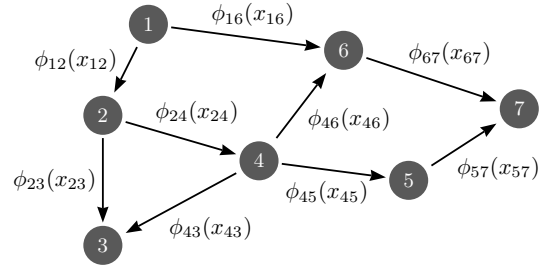

\mypar{Network flow}
	A network flow problem is typically formulated on a network with arcs (or directed edges), where an arc from node~$i$ to node~$j$ indicates a flow in that direction. In the example given in Fig.~\ref{Fig:NetworkFlow}, there can be a flow from node~$1$ to node~$6$, but not the opposite. Every arc~$(i,j) \in \mathcal{A}$ has associated a non-negative variable~$x_{ij}$ representing the amount of flow in that arc (from node~$i$ to node~$j$), and a cost function~$\phi_{ij}(x_{ij})$ that depends only on~$x_{ij}$. The goal is to minimize the sum of all the costs, while satisfying the laws of conservation of flow. External flow can be injected or extracted from a node, making that node a source or a sink, respectively. For example, in Fig.~\ref{Fig:NetworkFlow}, node~$1$ can only be a source, since it has only outward edges; in contrast, nodes~$3$ and~$7$ can only be sinks, since they have only inward edges. The remaining nodes may or may not be sources or sinks. We represent the network of flows with the node-arc incidence matrix~$B$, where the column associated to an arc from node~$i$ to node~$j$ has a~$-1$ in the $i$th entry, a $1$ in the $j$th entry, and zeros elsewhere. We assume the components of the variable~$x$ and the columns of~$B$ are in lexicographic order. For example, $x = (x_{12},x_{16},x_{23},x_{24},x_{43},x_{45},x_{46},x_{57},x_{67})$ would be the variable in Fig.~\ref{Fig:NetworkFlow}. The laws of conservation of flow are expressed as $B x = d$, where~$d \in \mathbb{R}^P$ is the vector of external inputs/outputs. The entries of~$d$ sum up to zero and~$d_p < 0$ (resp. $d_p > 0$) if node~$p$ is a source (resp. sink). When node~$p$ is neither a source nor a sink, $d_p = 0$. The problem we solve is
	\begin{equation}\label{Eq:NetworkFlow}
		\begin{array}{ll}
			\underset{x}{\text{minimize}} & \sum_{(i,j) \in \mathcal{A}} \phi_{ij}(x_{ij})\\
			\text{subject to} & Bx = d \\
			                  & x\geq 0\,,
		\end{array}
	\end{equation}
	which can be written as~\eqref{Eq:PartialProb} by setting
	\begin{multline*}
		f_p\Bigl(\{x_{pj}\}_{(p,j) \in \mathcal{A}},\{x_{jp}\}_{(j,p) \in \mathcal{A}}\Bigr) = \frac{1}{2}\sum_{(p,j) \in \mathcal{A}} \phi_{pj}(x_{pj}) \\+ \frac{1}{2}\sum_{(j,p) \in \mathcal{A}} \phi_{jp}(x_{jp}) + \text{i}_{b_p^\top x = d_p}(\{x_{pj}\}_{(p,j) \in \mathcal{A}},\{x_{jp}\}_{(j,p) \in \mathcal{A}})\,,
	\end{multline*}
	where~$b_p^\top$ is the $p$th row of~$B$. In words, $f_p$ consists of the sum of the functions associated to all arcs involving node~$p$, plus the indicator function of the set~$\{x\,:\,b_p^\top x = d_p\}$. This indicator function enforces the conservation of flow at node~$p$ and it only involves the variables~$\{x_{pj}\}_{(p,j) \in \mathcal{A}}$ and $\{x_{jp}\}_{(j,p) \in \mathcal{A}}$.
	
	Regarding the communication network~$\mathcal{G} = (\mathcal{V},\mathcal{E})$, we assume it consists of the underlying undirected network. This means that nodes~$i$ and~$j$ can exchange messages directly, i.e., $(i,j) \in \mathcal{E}$ for~$i<j$, if there is an arc between these nodes, i.e., $(i,j) \in \mathcal{A}$ or $(j,i) \in \mathcal{A}$. Therefore, in contrast with the flows, messages do not necessarily need to be exchanged satisfying the direction of the arcs. In fact, messages and flows might represent different physical quantities: think, for example, in a network of water pipes controlled by actuators at each pipe junction; while the pipes might enforce a direction in the flow of water (by using, for example, special valves), there is no reason to impose the same constraint on the electrical signals exchanged by the actuators. In problem~\eqref{Eq:NetworkFlow}, the subgraph induced by~$x_{ij}$, $(i,j) \in \mathcal{A}$, consists only of nodes~$i$ and~$j$ and an edge connecting them. This makes the variable in~\eqref{Eq:NetworkFlow} connected and star-shaped. Next we discuss the functions~$\phi_{ij}$ used in our simulations.

\mypar{Models for the experiments}
	We considered two instances of~\eqref{Eq:NetworkFlow}: a simple instance and a complex instance. While the simple instance makes all the algorithms we consider applicable, the (more) complex instance can be solved only by a subset of algorithms, but it provides a more realistic application. The simple instance uses $\phi_{ij} = \frac{1}{2}(x_{ij} - a_{ij})^2$, where~$a_{ij} > 0$, as the cost function for each arc~$(i,j)$ and no constraints besides the conservation of flow, i.e., we drop the nonnegativity constraint $x \geq 0$ in~\eqref{Eq:NetworkFlow}. The reason for dropping this constraint was to make the algorithm in~\cite{Zargham12-AcceleratedDualDescent} applicable. The other instance we consider is~\cite[Ch.17]{Ahuja93-NetworkFlows}:
  \begin{equation}\label{Eq:MultiComm2}
			\begin{array}{cl}
				\underset{x = \{x_{ij}\}_{(i,j) \in \mathcal{A}}}{\text{minimize}} & \sum_{(i,j) \in \mathcal{A}} \frac{x_{ij}}{c_{ij} - x_{ij}} \\
				\text{subject to} & B x = d \\
				                  & 0 \leq x_{ij} \leq c_{ij}\,, \quad (i,j) \in \mathcal{A}\,,
			\end{array}
	\end{equation}
	where~$c_{ij}$ represents the capacity of the arc $(i,j) \in \mathcal{A}$. Problem~\eqref{Eq:MultiComm2} has the same format of~\eqref{Eq:NetworkFlow} except for the additional capacity constraints~$x_{ij} \leq c_{ij}$, and it models overall system delays on multicommodity flow problems~\cite[Ch.17]{Ahuja93-NetworkFlows}.	
% 	Consider a network where~$K$ commodities will flow and the following problem:
%   \begin{equation}\label{Eq:MultiComm}
% 			\begin{array}{ll}
% 				\underset{x = \{x^k\}_{k=1}^K}{\text{minimize}} &  \sum_{(i,j) \in \mathcal{A}} \phi_{ij}(x^k) \\
% 				\text{subject to} & B x^k = d^k \,,\quad k = 1,\ldots,K \\
% 				                  & 0 \leq \sum_{k=1}^K x_{ij}^k \leq c_{ij}\,, \,\,\, k = 1,\ldots,K, \, (i,j) \in \mathcal{A}\,,
% 			\end{array}
% 	\end{equation}
% 	where~$x_{ij}^k$ is the flow of commodity~$k$ on arc~$(i,j) \in \mathcal{A}$ and $x^k = \{x_{ij}^k\}_{(i,j) \in \mathcal{A}}$ is the collection of flows of commodity~$k$ along all arcs. Also, $d^k$ is the vector of external inputs/outputs for commodity~$k$. Each arc~$(i,j) \in \mathcal{A}$ has a finite capacity~$c_{ij}$ such that the aggregate rate $\sum_{k=1}^K x_{ij}^k$ at that arc is smaller than $c_{ij}$. According to~\cite[Ch.17]{Ahuja93-NetworkFlows}, $\phi_{ij}(x^k) = \Bigl(\sum_{k=1}^K x_{ij}^k\Bigr)/(c_{ij} - \sum_{k=1}^K x_{ij}^k)$ gives a good model for the delay experienced by any commodity flowing through edge~$(i,j)$. Therefore, we can simplify~\eqref{Eq:MultiComm} by considering aggregate commodities, i.e., not distinguishing between different commodities. This simplification yields problem~\eqref{Eq:MultiComm2}, where $x_{ij} = \sum_{k=1}^K x_{ij}^k$, and $d  = \sum_{k=1}^K d^k$. 
	If we apply Algorithm~\ref{Alg:Conn} to problem~\eqref{Eq:MultiComm2}, node~$p$ has to solve at each step
	\begin{equation}\label{Eq:NetFlowEachNode}
			\begin{array}{cl}
				\underset{y= (y_1,\ldots,y_{D_p})}{\text{minimize}} & \sum_{i=1}^{D_p} (\frac{y_i}{c_i - y_i} + v_i y_i + a_i y_i^2)
				\\
				\text{subject to} & b_p^\top y = d_p \\
				                  & 0 \leq y \leq c\,,
			\end{array}
	\end{equation}
	where each~$y_i$ corresponds to~$x_{pj}$ if $(p,j) \in \mathcal{A}$, or to~$x_{jp}$ if~$(j,p) \in \mathcal{A}$. Since projecting a point onto the set of constraints of~\eqref{Eq:NetFlowEachNode} is simple (see~\cite{VandenbergheLecs}), \eqref{Eq:NetFlowEachNode} can be solved efficiently with a projected gradient method. In fact, we will use the algorithm in~\cite{Birgin00-NonmonotoneSpectralGradient}, which is based on the Barzilai-Borwein method.

	In sum, we will solve two instances of~\eqref{Eq:NetworkFlow}: a simple one, where~$\phi_{ij}(x_{ij}) = (1/2)(x_{ij} - a_{ij})^2$ and with no constraints besides~$Bx = d$, and~\eqref{Eq:MultiComm2}, a more complex but realistic one.

\section{Experimental Results}
\label{Sec:ExperimentalResults}

	In this section we show experimental results of the proposed algorithms solving MPC and network flow problems. We start with network flow because it is simpler and more algorithms are applicable. Also, it will illustrate the inefficiency of solving~\eqref{Eq:PartialProb} with an algorithm designed for the global problem~\eqref{Eq:SeparableOptim}.

	\begin{figure*}
     \centering
     \hspace{0.1cm}
     \subfigure[Simple instance of~\eqref{Eq:NetworkFlow}]{\label{SubFig:NF_Dummy}
     \begin{pspicture}(8.0,5.4)
       \rput[bl](0.25,0.70){\includegraphics[width=7.5cm]{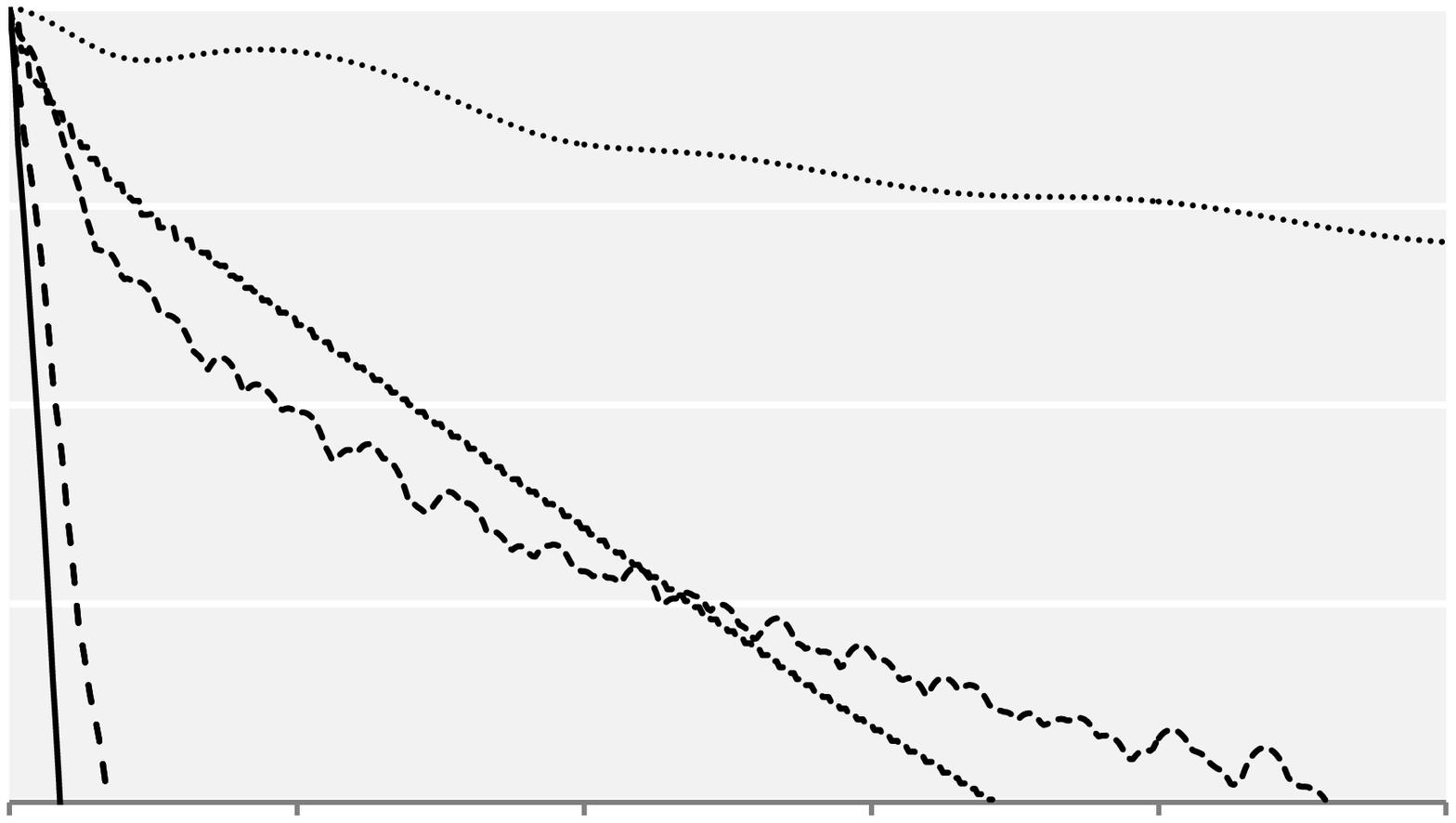}}
       \rput[b](4.00,0.001){\scriptsize \textbf{\sf Communication steps}}
       \rput[bl](-0.352,5.17){\mbox{\scriptsize \textbf{{\sf Relative error}}}}

       \rput[r](0.20,4.90){\scriptsize $\mathsf{10^{0\phantom{-}}}$}
       \rput[r](0.20,3.91){\scriptsize $\mathsf{10^{-1}}$}
       \rput[r](0.20,2.88){\scriptsize $\mathsf{10^{-2}}$}
       \rput[r](0.20,1.85){\scriptsize $\mathsf{10^{-3}}$}
       \rput[r](0.20,0.83){\scriptsize $\mathsf{10^{-4}}$}

       \rput[t](0.280,0.59){\scriptsize $\mathsf{0}$}
       \rput[t](1.763,0.59){\scriptsize $\mathsf{100}$}
       \rput[t](3.256,0.59){\scriptsize $\mathsf{200}$}
       \rput[t](4.750,0.59){\scriptsize $\mathsf{300}$}
       \rput[t](6.240,0.59){\scriptsize $\mathsf{400}$}
       \rput[t](7.738,0.59){\scriptsize $\mathsf{500}$}

       \rput[lb](1.21,1.14){\scriptsize \textbf{\sf Alg.\ref{Alg:Conn}}}
       \psline[linewidth=0.5pt](1.15,1.22)(0.55,1.1)
       \rput[lb](0.68,1.90){\scriptsize \textbf{\sf \cite{Kekatos12-DistributedRobustPowerStateEstimation,Boyd11-ADMM}}}
       \rput[lb](2.1,3.1){\scriptsize \textbf{\sf \cite{Zargham12-AcceleratedDualDescent}}}
			 \rput[lb](5.4,1.4){\scriptsize \textbf{\sf \cite{Nesterov03-Lectures}}}
			 \rput[lb](4.0,4.2){\scriptsize \textbf{\sf \cite{Mota12-DADMM}}}

       %\psgrid
     \end{pspicture}
     }
     \hfill
     \subfigure[Problem \eqref{Eq:MultiComm2}]{\label{SubFig:NF_delays}
     \begin{pspicture}(8.0,5.4)
       \rput[bl](0.25,0.70){\includegraphics[width=7.5cm]{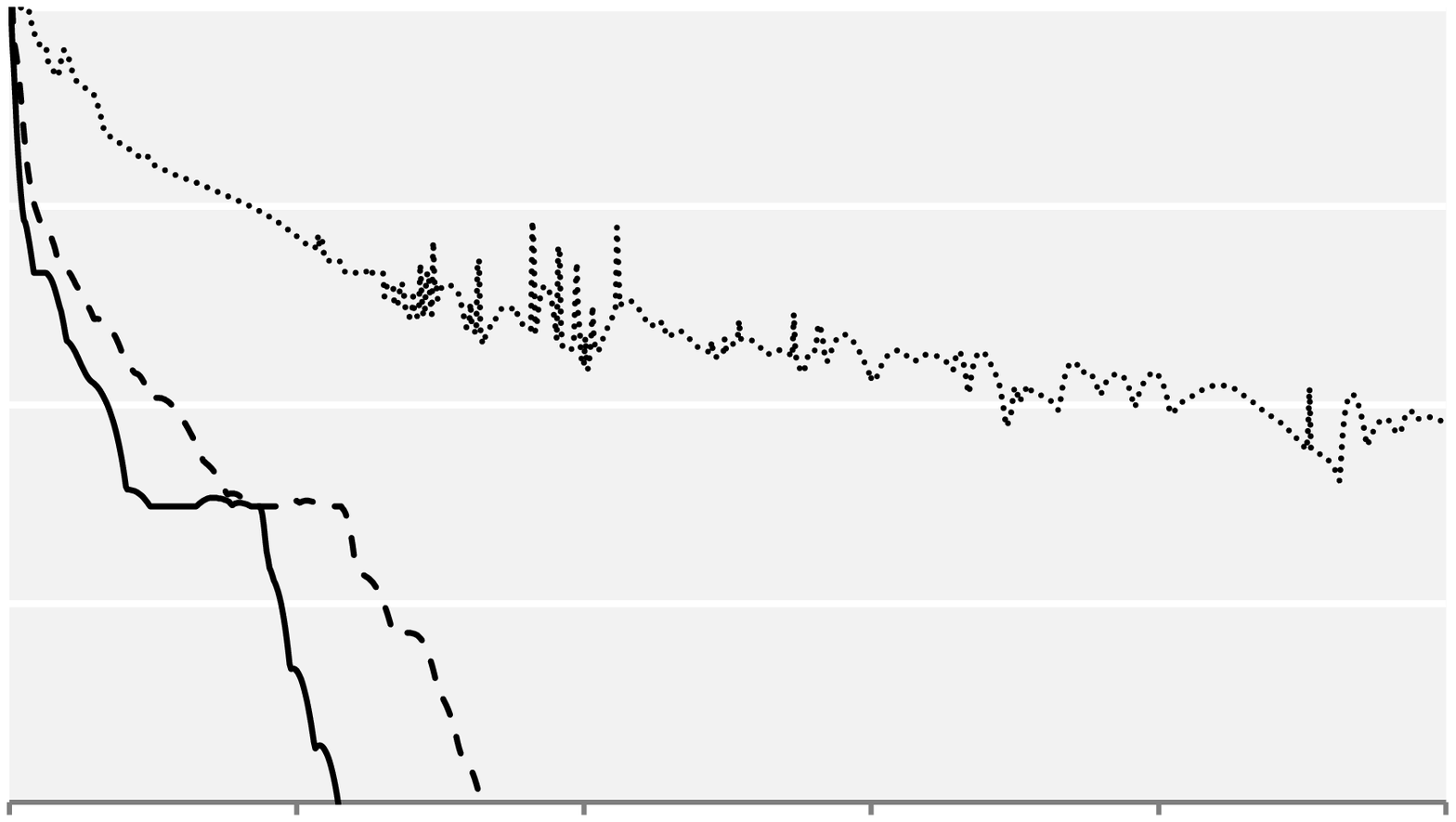}}
       \rput[b](4.00,0.001){\scriptsize \textbf{\sf Communication steps}}
       \rput[bl](-0.352,5.17){\mbox{\scriptsize \textbf{{\sf Relative error}}}}

       \rput[r](0.20,4.90){\scriptsize $\mathsf{10^{0\phantom{-}}}$}
       \rput[r](0.20,3.91){\scriptsize $\mathsf{10^{-1}}$}
       \rput[r](0.20,2.88){\scriptsize $\mathsf{10^{-2}}$}
       \rput[r](0.20,1.85){\scriptsize $\mathsf{10^{-3}}$}
       \rput[r](0.20,0.83){\scriptsize $\mathsf{10^{-4}}$}

       \rput[t](0.280,0.59){\scriptsize $\mathsf{0}$}
       \rput[t](1.763,0.59){\scriptsize $\mathsf{200}$}
       \rput[t](3.256,0.59){\scriptsize $\mathsf{400}$}
       \rput[t](4.750,0.59){\scriptsize $\mathsf{600}$}
       \rput[t](6.240,0.59){\scriptsize $\mathsf{800}$}
       \rput[t](7.738,0.59){\scriptsize $\mathsf{1000}$}

       \rput[rt](1.62,1.62){\scriptsize \textbf{\sf Alg.\ref{Alg:Conn}}}
       \rput[lb](2.10,2.20){\scriptsize \textbf{\sf \cite{Kekatos12-DistributedRobustPowerStateEstimation,Boyd11-ADMM}}}
       \rput[lb](4.40,3.40){\scriptsize \textbf{\sf \cite{Nesterov03-Lectures}}}

       %\psgrid
     \end{pspicture}
     }
     \hspace{0.1cm}
     \vspace{-0.1cm}
     \caption{
			 Results for the network flow problems on a network with~$2000$ nodes and~$3996$ edges. The results in~$\text{(a)}$ are for the simple instance of~\eqref{Eq:NetworkFlow}, where $\phi_{ij}(x_{ij}) = (1/2)(x_{ij} - a_{ij})^2$ and there are no nonnegativity constraints; and the results in~$\text{(b)}$ are for~\eqref{Eq:MultiComm2}.
		}
     \label{Fig:Exp_NF}
	\end{figure*}

\mypar{Network flows: experimental setup}
	As mentioned in the previous section, we solved two instances of~\eqref{Eq:NetworkFlow}. In both instances, we used a network with $2000$ nodes and $3996$ edges, generated randomly according to the Barabasi-Albert model~\cite{Barabasi99EmergenceOfScaling} with parameter~$2$, using the Network X Python package~\cite{NetworkX}. We made the simplifying assumption that between any two pairs of nodes there can be at most arc, as shown in Fig.~\ref{Fig:NetworkFlow}. Hence, the size of the variable~$x$ in~\eqref{Eq:NetworkFlow} is equal to the number of edges~$|\mathcal{E}|$, in this case, $3996$. The generated network had a diameter of~$8$, an average node degree of~$3.996$, and it was colored with~$3$ colors in Sage~\cite{sage}. This gives us the underlying (undirected) communication network. Then, we assigned randomly a direction to each edge, with equal probabilities for both directions, creating a directed network like in Fig.~\ref{Fig:NetworkFlow}. We also assigned to each edge a number drawn randomly from the set~$\{10, 20, 30, 40, 50, 100\}$. The probabilities were $0.2$ for the first four elements and~$0.1$ for~$50$ and~$100$. These numbers played the role of the $a_{ij}$'s in the simple instance of~\eqref{Eq:NetworkFlow} and the role of the capacities~$c_{ij}$ in~\eqref{Eq:MultiComm2}. To generate the vector~$d$ or, in other words, to determine which nodes are sources or sinks, we proceeded as follows. For each~$k = 1,\ldots,100$, we picked a source~$s_k$ randomly (uniformly) out of the set of~$2000$ nodes and then picked a sink~$r_k$ randomly (uniformly) out of the set of reachable nodes of~$s_k$. For example, if we were considering the network of Fig.~\ref{Fig:NetworkFlow} and picked~$s_k = 4$ as a source node, the set of its reachable nodes would be~$\{3, 5, 6, 7\}$. Next, we added to the entries~$s_k$ and~$r_k$ of~$d$ the values $-f_k/100$ and~$f_k/100$, respectively, where~$f_k$ is a number drawn randomly exactly as~$c_{ij}$ (or~$a_{ij}$). This corresponds to injecting a flow of quantity~$f_k/100$ at node~$s_k$ and extracting the same quantity at node~$r_k$. After repeating this process~$100$ times, for~$k = 1,\ldots,K$, we obtained vector~$d$.

	To assess the error given by each algorithm, we computed the solutions~$x^\star$ of the instances of~\eqref{Eq:NetworkFlow} in a centralized way. 	The simple instance of~\eqref{Eq:NetworkFlow} considers~$\phi_{ij}(x_{ij}) = (1/2)(x_{ij} - a_{ij})^2$ and ignores the constraint~$x\geq 0$. Thus, it is a simple quadratic program and has a closed-form solution: solving a linear system. Similarly, the problem Algorithm~\ref{Alg:Conn} (resp. Algorithm~\ref{Alg:Kekatos}) has to solve in step~\ref{SubAlg:Conn_Prob} (resp. step~\ref{SubAlg:2BlockPartialProb}) boils down to solving a linear system. To compute the solution of~\eqref{Eq:MultiComm2}, the complex instance of~\eqref{Eq:NetworkFlow}, we used CVXOPT~\cite{CVXOPT}.

	The plots we will show depict the relative error on the primal variable $\|x^k - x^\star\|_{\infty}/\|x^\star\|_{\infty}$, where~$x^k$ is the concatenation of the estimates at all nodes, versus the number of communication steps. A \textit{communication step} (CS) consists of all nodes communicating their current estimates to their neighbors. That is, in each CS, information flows on each edge in both directions and, hence, the total number of CSs is proportional to the total number of communications. All the algorithms we compared, discussed next, had a tuning parameter: $\rho$ for the ADMM-based algorithms (cf. Algorithms~\ref{Alg:Conn} and~\ref{Alg:Kekatos}), a Lipschitz constant~$L$ for a gradient-based algorithm, and a stepsize~$\alpha$ for a Newton-based algorithm. Suppose we selected~$\bar{\rho}$ for an ADMM-based algorithm. We say that~$\bar{\rho}$ \textit{has precision}~$\gamma$, if both $\bar{\rho} - \gamma$ and $\bar{\rho} + \gamma$ lead to worse results for that algorithm. A similar definition is used for~$L$ and~$\alpha$. We compared Algorithm~\ref{Alg:Conn}, henceforth denoted as Alg.~\ref{Alg:Conn}, against the ADMM-based algorithms in~\cite[\S7.2]{Boyd11-ADMM} and~\cite{Kekatos12-DistributedRobustPowerStateEstimation} (recall that Algorithm~\ref{Alg:Kekatos} describes~\cite{Kekatos12-DistributedRobustPowerStateEstimation}), Nesterov's method~\cite{Nesterov03-Lectures}, the distributed Newton method proposed in~\cite{Zargham12-AcceleratedDualDescent}, and D-ADMM~\cite{Mota12-DADMM}. For network flow problems, the algorithms in~\cite[\S7.2]{Boyd11-ADMM} and~\cite{Kekatos12-DistributedRobustPowerStateEstimation} coincide, i.e., they become exactly the same algorithm. This is not surprising since both are based on the same algorithm: the $2$-block ADMM. All the ADMM-based algorithms, including Alg.~\ref{Alg:Conn}, take~$1$ CS per iteration. The work in~\cite{Zargham12-AcceleratedDualDescent}, besides proposing a distributed Newton method, also describes the application of the gradient method to the dual of~\eqref{Eq:NetworkFlow}. Here, instead of applying the simple gradient method, we apply Nesterov's method~\cite{Nesterov03-Lectures}, which can be applied in the same conditions, has a better bound on the convergence rate, and is known to converge faster in practice. However, gradient methods, including Nesterov's method, require an objective that has a Lipschitz-continuous gradient. While this is the case of the objective of the dual of the simple instance of~\eqref{Eq:NetworkFlow}, the same does not happen for the objective of the dual of~\eqref{Eq:MultiComm2}. Therefore, in the latter case, we had to estimate a Lipschitz constant~$L$. Similarly to the ADMM-based algorithms, each iteration of a gradient algorithm takes $1$ CS per iteration. Regarding the distributed Newton algorithm in~\cite{Zargham12-AcceleratedDualDescent}, we implemented it with a parameter~$N=2$, which is the order of the approximation in the computation of the Newton direction, and fixed the stepsize~$\alpha$. With this implementation, each iteration takes $3$ CSs. Finally, D-ADMM~\cite{Mota12-DADMM} is currently the most communication-efficient algorithm for the global problem~\eqref{Eq:SeparableOptim}. As such, it makes all the nodes compute the full solution~$x^\star$, which has dimensions~$3996$ in this case. Thus, each message exchanged in one CS of D-ADMM is~$3996$ times larger than the messages exchanged by the other algorithms.

\mypar{Network flows: results}
	The results for the simple instance of~\eqref{Eq:NetworkFlow} are shown in Fig.~\ref{SubFig:NF_Dummy}. Of all the algorithms, Alg.~\ref{Alg:Conn} required the least amount of CSs to achieve any relative error between~$1$ and~$10^{-4}$. The second best were the algorithms~\cite{Boyd11-ADMM} and \cite{Kekatos12-DistributedRobustPowerStateEstimation}, whose lines coincide because they become the same algorithm when applied to network flows. Nesterov's method~\cite{Nesterov03-Lectures} and the Newton-based method~\cite{Zargham12-AcceleratedDualDescent} had a performance very similar to each other, but worse than the ADMM-based algorithms. However, D-ADMM~\cite{Mota12-DADMM}, which is also ADMM-based but solves the global problem~\eqref{Eq:SeparableOptim} instead, was the algorithm with the worst performance. Note that, in addition to requiring much more CSs than any other algorithm, each message exchanged by~\cite{Mota12-DADMM} is $3996$ times larger than a message exchanged by any other algorithm. This clearly shows that if we want to derive communication-efficient algorithms, we have to explore the structure of~\eqref{Eq:SeparableOptim}. Finally, we mention that the value of~$\rho$ in these experiments was~$2$ for all ADMM-based algorithms (precision~$1$), the Lipschitz constant~$L$ was~$70$ (precision~$5$), and the stepsize $\alpha$ was~$0.4$ (precision~$0.1$).

	Fig.~\ref{SubFig:NF_delays} shows the results for~\eqref{Eq:MultiComm2}. In this case,
	we were not able to make the algorithm in~\cite{Zargham12-AcceleratedDualDescent} converge (actually, it is not guaranteed to converge for this problem). It is visible in Fig.~\ref{SubFig:NF_delays} that this problem is harder to solve, since all algorithms required more CSs solve it. Again, Alg.~\ref{Alg:Conn} was the algorithm with the best performance. This time we did not find any choice for~$L$ that made Nesterov's algorithm~\cite{Nesterov03-Lectures}  achieve an error of~$10^{-4}$ in less than~$1000$ CSs. The best result we obtained was for~$L = 15000$. The parameter~$\rho$ was~$0.08$ for Alg.~\ref{Alg:Conn} and~$0.12$ for~\cite{Kekatos12-DistributedRobustPowerStateEstimation,Boyd11-ADMM}, both computed with precision~$0.02$.

	\begin{figure*}
     \centering
     \hspace{0.1cm}
     \subfigure[Network A with star-shaped induced subgraphs]{\label{SubFig:MPC_Barab100_Unstable}
     \begin{pspicture}(8.0,5.4)
       \rput[bl](0.25,0.70){\includegraphics[width=7.5cm]{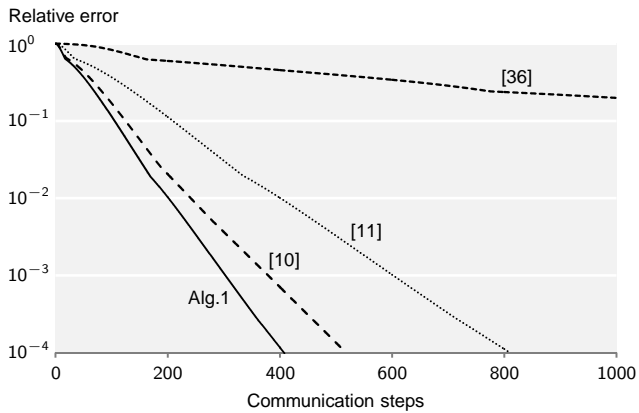}}
       \rput[b](4.00,0.001){\scriptsize \textbf{\sf Communication steps}}
       \rput[bl](-0.352,5.17){\mbox{\scriptsize \textbf{{\sf Relative error}}}}

       \rput[r](0.20,4.89){\scriptsize $\mathsf{10^{0\phantom{-}}}$}
       \rput[r](0.20,3.90){\scriptsize $\mathsf{10^{-1}}$}
       \rput[r](0.20,2.87){\scriptsize $\mathsf{10^{-2}}$}
       \rput[r](0.20,1.84){\scriptsize $\mathsf{10^{-3}}$}
       \rput[r](0.20,0.81){\scriptsize $\mathsf{10^{-4}}$}

       \rput[t](0.280,0.59){\scriptsize $\mathsf{0}$}
       \rput[t](1.764,0.59){\scriptsize $\mathsf{200}$}
       \rput[t](3.261,0.59){\scriptsize $\mathsf{400}$}
       \rput[t](4.749,0.59){\scriptsize $\mathsf{600}$}
       \rput[t](6.239,0.59){\scriptsize $\mathsf{800}$}
			 \rput[t](7.734,0.59){\scriptsize $\mathsf{1000}$}

			 \rput[rt](2.60,1.60){\scriptsize \textbf{\sf Alg.\ref{Alg:Conn}}}
       \rput[bl](3.11,1.88){\scriptsize \textbf{\sf \cite{Boyd11-ADMM}}}
       \rput[bl](4.20,2.27){\scriptsize \textbf{\sf \cite{Kekatos12-DistributedRobustPowerStateEstimation}}}
       \rput[lb](6.20,4.32){\scriptsize \textbf{\sf \cite{Nesterov03-Lectures}}}

       %\psgrid
     \end{pspicture}
     }
     \hfill
     \subfigure[Network B with star-shaped induced subgraphs]{\label{SubFig:MPC_PowerGridStable}
     \begin{pspicture}(8.0,5.4)
       \rput[bl](0.25,0.70){\includegraphics[width=7.5cm]{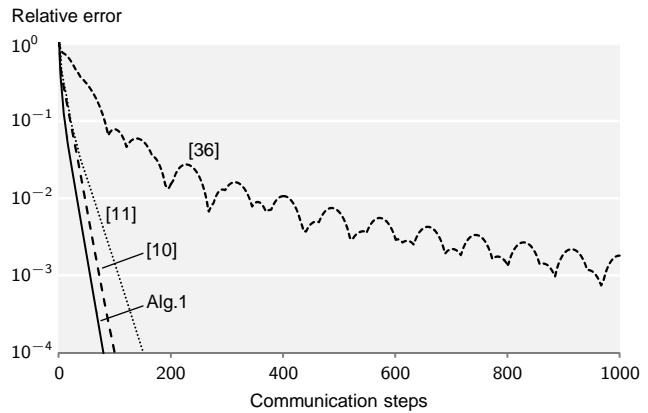}}
       \rput[b](4.00,0.001){\scriptsize \textbf{\sf Communication steps}}
       \rput[bl](-0.352,5.17){\mbox{\scriptsize \textbf{{\sf Relative error}}}}

       \rput[r](0.20,4.89){\scriptsize $\mathsf{10^{0\phantom{-}}}$}
       \rput[r](0.20,3.90){\scriptsize $\mathsf{10^{-1}}$}
       \rput[r](0.20,2.87){\scriptsize $\mathsf{10^{-2}}$}
       \rput[r](0.20,1.84){\scriptsize $\mathsf{10^{-3}}$}
       \rput[r](0.20,0.81){\scriptsize $\mathsf{10^{-4}}$}

       \rput[t](0.280,0.59){\scriptsize $\mathsf{0}$}
       \rput[t](1.764,0.59){\scriptsize $\mathsf{200}$}
       \rput[t](3.261,0.59){\scriptsize $\mathsf{400}$}
       \rput[t](4.749,0.59){\scriptsize $\mathsf{600}$}
       \rput[t](6.239,0.59){\scriptsize $\mathsf{800}$}
			 \rput[t](7.734,0.59){\scriptsize $\mathsf{1000}$}

       \rput[bl](1.44,1.32){\scriptsize \textbf{\sf Alg.\ref{Alg:Conn}}}
       \psline[linewidth=0.5pt](1.40,1.40)(0.83,1.20)
       \rput[bl](1.44,2.02){\scriptsize \textbf{\sf \cite{Boyd11-ADMM}}}
       \psline[linewidth=0.5pt](1.4,2.1)(0.85,1.9)
       \rput[lb](0.90,2.50){\scriptsize \textbf{\sf \cite{Kekatos12-DistributedRobustPowerStateEstimation}}}
       \rput[lb](2.0,3.35){\scriptsize \textbf{\sf \cite{Nesterov03-Lectures}}}

       %\psgrid
     \end{pspicture}
     }
     \hspace{0.1cm}

     \hspace{0.1cm}
     \hfill
     \subfigure[Network A with a generic connected variable]{\label{SubFig:MPC_Barab100NonStarUnstable}
     \begin{pspicture}(8.0,5.4)
       \rput[bl](0.25,0.70){\includegraphics[width=7.5cm]{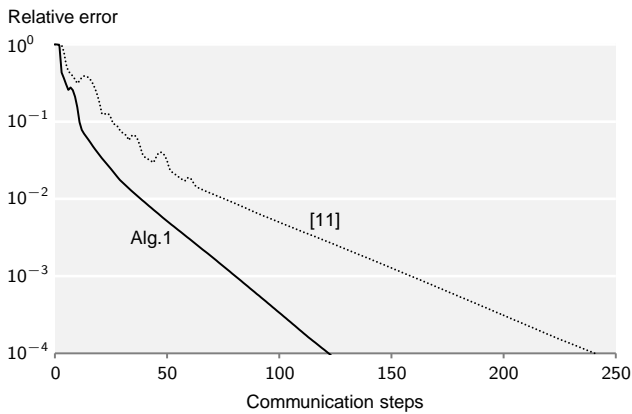}}
       \rput[b](4.00,0.001){\scriptsize \textbf{\sf Communication steps}}
       \rput[bl](-0.352,5.17){\mbox{\scriptsize \textbf{{\sf Relative error}}}}

       \rput[r](0.20,4.89){\scriptsize $\mathsf{10^{0\phantom{-}}}$}
       \rput[r](0.20,3.90){\scriptsize $\mathsf{10^{-1}}$}
       \rput[r](0.20,2.87){\scriptsize $\mathsf{10^{-2}}$}
       \rput[r](0.20,1.84){\scriptsize $\mathsf{10^{-3}}$}
       \rput[r](0.20,0.81){\scriptsize $\mathsf{10^{-4}}$}

       \rput[t](0.280,0.59){\scriptsize $\mathsf{0}$}
       \rput[t](1.768,0.59){\scriptsize $\mathsf{50}$}
       \rput[t](3.261,0.59){\scriptsize $\mathsf{100}$}
       \rput[t](4.749,0.59){\scriptsize $\mathsf{150}$}
       \rput[t](6.237,0.59){\scriptsize $\mathsf{200}$}
			 \rput[t](7.732,0.59){\scriptsize $\mathsf{250}$}

			 \rput[rt](1.84,2.40){\scriptsize \textbf{\sf Alg.\ref{Alg:Conn}}}
       \rput[bl](3.66,2.42){\scriptsize \textbf{\sf \cite{Kekatos12-DistributedRobustPowerStateEstimation}}}

       %\psgrid
     \end{pspicture}
     }
     \hfill
     \subfigure[Network B with a generic connected variable]{\label{SubFig:MPC_PowerGridStableNonStar}
     \begin{pspicture}(8.0,5.4)
       \rput[bl](0.25,0.70){\includegraphics[width=7.5cm]{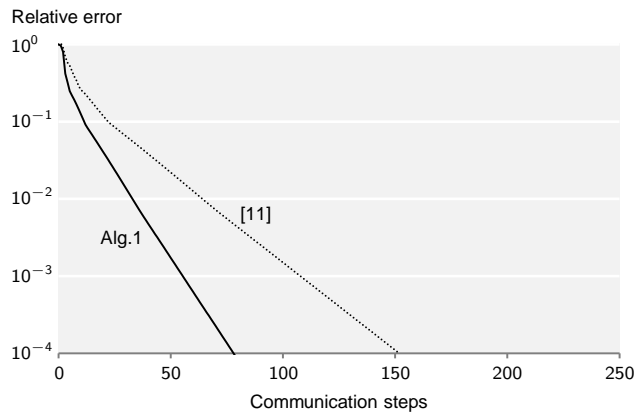}}
       \rput[b](4.00,0.001){\scriptsize \textbf{\sf Communication steps}}
       \rput[bl](-0.352,5.17){\mbox{\scriptsize \textbf{{\sf Relative error}}}}

       \rput[r](0.20,4.89){\scriptsize $\mathsf{10^{0\phantom{-}}}$}
       \rput[r](0.20,3.90){\scriptsize $\mathsf{10^{-1}}$}
       \rput[r](0.20,2.87){\scriptsize $\mathsf{10^{-2}}$}
       \rput[r](0.20,1.84){\scriptsize $\mathsf{10^{-3}}$}
       \rput[r](0.20,0.81){\scriptsize $\mathsf{10^{-4}}$}

       \rput[t](0.280,0.59){\scriptsize $\mathsf{0}$}
       \rput[t](1.768,0.59){\scriptsize $\mathsf{50}$}
       \rput[t](3.261,0.59){\scriptsize $\mathsf{100}$}
       \rput[t](4.749,0.59){\scriptsize $\mathsf{150}$}
       \rput[t](6.237,0.59){\scriptsize $\mathsf{200}$}
			 \rput[t](7.732,0.59){\scriptsize $\mathsf{250}$}

			 \rput[rt](1.39,2.40){\scriptsize \textbf{\sf Alg.\ref{Alg:Conn}}}
       \rput[bl](2.69,2.50){\scriptsize \textbf{\sf \cite{Kekatos12-DistributedRobustPowerStateEstimation}}}

       %\psgrid
     \end{pspicture}
     }
     \hspace{0.1cm}

     \vspace{-0.1cm}
     \caption{Results for MPC. The variable is connected in all cases, i.e., the subgraphs induced by all the components are connected. While in $\text{(a)}$ and $\text{(b)}$ each induced subgraph is a star, i.e., the interactions occur only between neighboring subsystems, in $\text{(c)}$ and $\text{(d)}$ each induced subgraph is generic. Only Algorithms~\ref{Alg:Conn} and~\ref{Alg:Kekatos} (\cite{Kekatos12-DistributedRobustPowerStateEstimation}) are applicable in the latter case. Alg.~\ref{Alg:Conn} was always the algorithm requiring the least number of communication steps to converge.}
     \label{Fig:Exp_MPC}
	\end{figure*}

	\begin{table}
    \centering
    \caption{
             Statistics for the networks used in MPC.
            }
    \label{Tab:Networks}
    \smallskip
        \renewcommand{\arraystretch}{1.3}
        %\resizebox{0.7\linewidth}{!}{
        \begin{tabular}{@{}ccccccc@{}}
          \toprule[1pt]
          Name & Source & \# Nodes & \# Edges & Diam. & \# Colors & Av. Deg.\\
          \midrule
           A & \cite{Barabasi99EmergenceOfScaling}         & $\phantom{4}100$  & $\phantom{6}196$  & $\phantom{4}6$  & $3$ & $3.92$ \\
           B & \cite{Watts98CollectiveDynamicsSmallWorld}  & $4941$ & $6594$ & $46$ & $6$ & $2.67$ \\
          \bottomrule[1pt]
        \end{tabular}
        %}
  \end{table}

\mypar{MPC: experimental setup}
	For the MPC experiments we used two networks with very different sizes. One network, which we call A, has~$100$ nodes, $196$ edges, and was generated the same way as the network for the network flow experiments: with a Barabasi-Albert model~\cite{Barabasi99EmergenceOfScaling} with parameter~$2$. The other network, named B, has $4941$ nodes and~$6594$ edges and it represents the topology of the Western States Power Grid~\cite{Watts98CollectiveDynamicsSmallWorld} (obtained in~\cite{NewmanPowerGrid}). The diameter, the number of used colors, and the average degree for these networks is shown in Table~\ref{Tab:Networks}. For coloring the networks, we used Sage~\cite{sage}.

	We solved the MPC problem~\eqref{Eq:DistributedMPC_SimpleModelFinal} and, to illustrate all the particular cases of a variable for~\eqref{Eq:PartialProb}, we created several types of data. For all the data types, the size of the state (resp. input) at each node was always~$n_p = 3$ (resp. $m_p=1$), and the time-horizon was~$T = 5$. Since~\eqref{Eq:DistributedMPC_SimpleModelFinal} has a variable of size~$m_p T P$, network A implied a variable of size~$500$ and network B implied a variable of size~$24705$. With network~A, we generated the matrices~$A_p$ so that each subsystem could be unstable; namely, we drew each of its entries from a normal distribution. With network~B, we proceeded the same way, but then ``shrunk'' the eigenvalues of each~$A_p$ to the interval $[-1,1]$, hence making each subsystem stable. All matrices~$B_{pj}$ were always generated as each~$A_p$ in the unstable case. The way we generated system couplings, i.e., the set~$\Omega_p$ for each node~$p$ (see also the dotted arrows in the networks of Fig.~\ref{Fig:MPC}), will be explained as we present the experimental results. Note that for the MPC problem~\eqref{Eq:DistributedMPC_SimpleModelFinal} the Lipschitz constant of the gradient of its objective can be computed in closed-form and, therefore, does not need to be estimated. The relative error will be computed as in the network flows: $\|x^k - x^\star\|_{\infty}/\|x^\star\|_{\infty}$, where~$x^k$ is the concatenation of all the nodes' input estimates.

\mypar{MPC results: connected case}
	The results for all the experiments on a connected variable are shown in Fig.~\ref{Fig:Exp_MPC}. There, Alg.~\ref{Alg:Conn} is compared against~\cite{Kekatos12-DistributedRobustPowerStateEstimation} (see also Algorithm~\ref{Alg:Kekatos}), and~\cite{Boyd11-ADMM}, and~\cite{Nesterov03-Lectures}. We mention that algorithms~\cite{Boyd11-ADMM,Nesterov03-Lectures} were already applied to~\eqref{Eq:DistributedMPC_SimpleModelFinal}, e.g., in~\cite{Conte12-ComputationalAspectsDistributedMPC}, in the special case of a variable with star-shaped induced subgraphs. This is in fact the only case where~\cite{Boyd11-ADMM} and~\cite{Nesterov03-Lectures} are distributed, and it explains why they are not in Figs.~\ref{SubFig:MPC_Barab100NonStarUnstable} and~\ref{SubFig:MPC_PowerGridStableNonStar}: the induced subgraphs in those figures are not stars. Only Alg.~\ref{Alg:Conn} and~\cite{Kekatos12-DistributedRobustPowerStateEstimation} are applicable in this case.

	In Fig.~\ref{SubFig:MPC_Barab100_Unstable} the network is~A and each subsystem was generated (possibly) unstable, and in Fig.~\ref{SubFig:MPC_PowerGridStable} the network is~B and each subsystem was generated stable. In both cases, Alg.~\ref{Alg:Conn} required the least number of CSs to achieve any relative error between $1$ and~$10^{-4}$, followed by~\cite{Boyd11-ADMM}, then by~\cite{Kekatos12-DistributedRobustPowerStateEstimation}, and finally by~\cite{Nesterov03-Lectures}. It can be seen from these plots that the difficulty of the problem is determined, not so much by the size of network, but by the stability of the subsystems. In fact, all algorithms required uniformly more communications to solve a problem on network~A, which has only~$100$ nodes, than on network~B, which has approximately $5000$ nodes. This difficulty can be measured by the Lipschitz constant~$L$: $1.63 \times 10^6$ for network~A (Fig.~\ref{SubFig:MPC_Barab100_Unstable}) and $3395$ for network~B (Fig.~\ref{SubFig:MPC_PowerGridStable}). Regarding the parameter~$\rho$, in Fig.~\ref{SubFig:MPC_Barab100_Unstable} it was $120$ for~\cite{Boyd11-ADMM} and~$135$ for the other algorithms (computed with precision $5$); in Fig.~\ref{SubFig:MPC_PowerGridStable}, it was~$25$ for Alg.~\ref{Alg:Conn} and~\cite{Boyd11-ADMM}, and~$30$ for Alg.~\cite{Kekatos12-DistributedRobustPowerStateEstimation} (also computed with precision $5$).

	In Figs.~\ref{SubFig:MPC_Barab100NonStarUnstable} and~\ref{SubFig:MPC_PowerGridStableNonStar} we considered a generic connected variable, where each induced subgraphs is not necessarily a star. In this case, the system couplings were generated as follows. Given a node~$p$, we assigned it~$u_p$ and we initialized a fringe with its neighbors~$\mathcal{N}_p$. Then, we selected a node randomly (with equal probability) from the fringe and made it depend on~$u_p$; we also added its neighbors to the fringe. The described process was done~$3$ times for each variable~$u_p$ (i.e., node~$p$). When each induced subgraph is not a star, only Alg.~\ref{Alg:Conn} and~\cite{Kekatos12-DistributedRobustPowerStateEstimation} are applicable. Figs.~\ref{SubFig:MPC_Barab100NonStarUnstable} and~\ref{SubFig:MPC_PowerGridStableNonStar} show their performance for network~A with unstable subsystems and for network~B with stable subsystems, respectively. It can be seen that Alg.~\ref{Alg:Conn} required uniformly less CSs than~\cite{Kekatos12-DistributedRobustPowerStateEstimation} to achieve the same relative error.

	\begin{figure}
     \centering
     \begin{pspicture}(8.0,5.4)
       \rput[bl](0.25,0.70){\includegraphics[width=7.5cm]{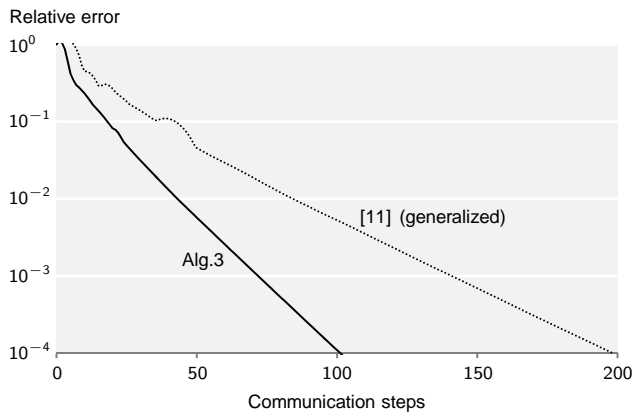}}
       \rput[b](4.00,0.001){\scriptsize \textbf{\sf Communication steps}}
       \rput[bl](-0.352,5.17){\mbox{\scriptsize \textbf{{\sf Relative error}}}}

       \rput[r](0.20,4.89){\scriptsize $\mathsf{10^{0\phantom{-}}}$}
       \rput[r](0.20,3.90){\scriptsize $\mathsf{10^{-1}}$}
       \rput[r](0.20,2.87){\scriptsize $\mathsf{10^{-2}}$}
       \rput[r](0.20,1.84){\scriptsize $\mathsf{10^{-3}}$}
       \rput[r](0.20,0.81){\scriptsize $\mathsf{10^{-4}}$}

       \rput[t](0.280,0.59){\scriptsize $\mathsf{0}$}
       \rput[t](2.150,0.59){\scriptsize $\mathsf{50}$}
       \rput[t](4.000,0.59){\scriptsize $\mathsf{100}$}
       \rput[t](5.878,0.59){\scriptsize $\mathsf{150}$}
       \rput[t](7.734,0.59){\scriptsize $\mathsf{200}$}

       \rput[rt](2.50,2.11){\scriptsize \textbf{\sf Alg.\ref{Alg:NonConn}}}
       \rput[lb](4.31,2.45){\scriptsize \textbf{\sf \cite{Kekatos12-DistributedRobustPowerStateEstimation} \,(generalized)}}

       %\psgrid
     \end{pspicture}

     \vspace{-0.1cm}
     \caption{Results for MPC when the variable is non-connected. The communication network is A and all the subsytems were designed stable.}
     \label{Fig:Exp_MPC_NonConnected}
	\end{figure}

\mypar{MPC results: non-connected case}
	A non-connected variable has at least one component whose induced subgraph~$\mathcal{G}_l = (\mathcal{V}_l,\mathcal{E}_l)$ is not connected. In this case, Algorithm~\ref{Alg:Conn} is no longer applicable and it requires a generalization, shown in Algorithm~\ref{Alg:NonConn}. Part of the generalization consists of computing Steiner trees, using the nodes in~$\mathcal{V}_l$ as required nodes. The same generalization can be made to the algorithm in~\cite{Kekatos12-DistributedRobustPowerStateEstimation}.

	To create a problem instance with a non-connected variable, we generated system couplings in a way very similar to the couplings for Figs.~\ref{SubFig:MPC_Barab100NonStarUnstable} and~\ref{SubFig:MPC_PowerGridStableNonStar}. The difference was that any node in the network could be chosen to depend on a given~$u_p$. However, any node in the fringe had twice the probability of being chosen than any other node. This process was run on network A for each one of its~$500$ components (recall that the variable size for network A is $500$), and obtained~$400$ non-connected components, i.e., $400$ components whose induced subgraphs were not connected. Then, as described in the preprocessing part of Algorithm~\ref{Alg:NonConn}, we computed Steiner trees for each non-connected component: $44\%$ of the nodes were Steiner for at least one component. To compute Steiner trees, we used a built-in Sage function~\cite{sage}. In this case, we generated all the subsystems stable. Then, we ran Algorithms~\ref{Alg:NonConn} and~\cite{Kekatos12-DistributedRobustPowerStateEstimation} (with a similar generalization) with $\rho = 35$ (computed with precision~$5$ for both algorithms). The results of these experiments are in Fig.~\ref{Fig:Exp_MPC_NonConnected}. Again, Algorithm~\ref{Alg:NonConn} required uniformly less CSs to converge than our generalization of~\cite{Kekatos12-DistributedRobustPowerStateEstimation}.

\section{Conclusions}
	We solved a class of optimization problems with the following structure: no component of the optimization variable appears in the functions of all nodes. Our approach considers two different cases, a connected and a non-connected variable, and proposes an algorithm for each. Our algorithms require a coloring scheme of the network and their convergence is guaranteed only for the special case of a bipartite network or for problems with strongly convex objectives. However, in the practical examples that we considered, the algorithm converges even when none of these conditions is met. Moreover, experimental results show that our algorithms require less communications to solve a given network flow or MPC problem to an arbitrary level of accuracy than prior algorithms.

\bibliographystyle{IEEEbib}

{ \isdraft{\singlespace}{}
\bibliography{paper}
}

\appendices

\section{Proof of Lemma~\ref{Lem:EquivalenceNodesC1}}
\label{App:EquivalenceNodesC1}

	To go from~\eqref{Eq:PartialC1_1} to~\eqref{Eq:PartialC1_2}, we first develop the last two terms of~\eqref{Eq:PartialC1_1}, respectively,
	\begin{equation}\label{App:Eq_1}
		{\lambda^k}^\top \bar{A}^1 \bar{x}^1
	\end{equation}
	and
	\begin{equation}\label{App:Eq_2}
		\frac{\rho}{2}\Bigl\|\bar{A}^1 \bar{x}^{1} + \sum_{c = 2}^C \bar{A}^c \bar{x}^{c,k}\Bigr\|^2\,.
	\end{equation}
	We first address~\eqref{App:Eq_1}. 	Given the structure of~$\bar{A}^1$, as seen in \eqref{Eq:PartialStructureMatrices}, we can write~\eqref{App:Eq_1} as $\sum_{l=1}^n ((\bar{A}_l^1)^\top \lambda_l^k)^\top \bar{x}_l^1$.	Recall that~$(\bar{A}_l^1)^\top$, if it exists (i.e., if there is a node with color~$1$ that depends on component~$x_l$), consists of the block of rows of the node-arc incidence matrix of~$\mathcal{G}_l$ corresponding to the nodes with color~$1$. Therefore, if there exists $p \in \mathcal{C}_1 \cap \mathcal{V}_l$, the vector~$(\bar{A}_l^1)^\top \lambda_l^k$ will have an entry $\sum_{j \in \mathcal{N}_p \cap \mathcal{V}_l} \text{sign}(j - p)\lambda_l^{pj,k}$. The sign function appears here because the column of the node-arc incidence matrix corresponding to $x_l^{(i)} - x_l^{(j)} = 0$, for a pair~$(i,j) \in \mathcal{E}_l$, contains $1$ in the $i$th entry and~$-1$ in the $j$th entry, where~$i<j$. In the previous expression, we used an extension of the definition of~$\lambda_l^{ij}$, which was only defined for~$i<j$ (due to our convention that for any edge~$(i,j) \in \mathcal{E}$ we have always~$i<j$). Assume~$\lambda_l^{ij}$ is initialized with zero; switching~$i$ and~$j$ in~\eqref{Eq:PartialDualVariableUpdate}, we obtain~$\lambda_l^{ji,k} = -\lambda_l^{ij,k}$, which holds for all iterations~$k$. To be consistent with the previous equation, we define~$\lambda_l^{ij}$ as $\lambda_l^{ij} := -\lambda_l^{ji}$ whenever~$i>j$. Therefore, \eqref{App:Eq_1} develops as
	\begin{align}
	    {\lambda^k}^\top \bar{A}^1 \bar{x}^1
	  &=
		  \sum_{l=1}^n ((\bar{A}_l^1)^\top \lambda_l^k)^\top \bar{x}_l^1
		  \notag
		\\
		&=
		  \sum_{l=1}^n \sum_{p \in \mathcal{C}_1} \sum_{j \in \mathcal{N}_p \cap \mathcal{V}_l} \!\!\! \text{sign}(j-p) \Bigl(\lambda_l^{pj,k}\Bigr)^\top x_l^{(p)}
		\notag
		\\
		&=
		  \sum_{p \in \mathcal{C}_1} \sum_{l=1}^n \sum_{j \in \mathcal{N}_p \cap \mathcal{V}_l} \!\!\! \text{sign}(j-p) \Bigl(\lambda_l^{pj,k}\Bigr)^\top x_l^{(p)}\,.
		\label{Eq:AppFirstTerm_1}
	\end{align}

	Regarding~\eqref{App:Eq_2}, it can be written as
	\begin{align}
	      &\frac{\rho}{2}\Bigl\|\bar{A}^1 \bar{x}^1 + \sum_{c = 2}^C \bar{A}^c \bar{x}^{c,k}\Bigr\|^2
		  \notag
		  \\
		  &=
		    \frac{\rho}{2}\Bigl\|\bar{A}^1 \bar{x}^1\Bigr\|^2 + \rho (\bar{A}^1\bar{x}^1)^\top \sum_{c=2}^C \bar{A}^c \bar{x}^{c,k}
		    + \frac{\rho}{2}\Bigl\|\sum_{c=2}^C \bar{A}^c \bar{x}^{c,k}\Bigr\|^2\,.
		  \label{Eq:AppSecondTerm_1}
	\end{align}
	Since the last term does not depend on~$\bar{x}^1$, it can be dropped from the optimization problem. We now use the structure of~$\bar{A}^1$ to rewrite the first term of~\eqref{Eq:AppSecondTerm_1}:
	\begin{align}
		  \frac{\rho}{2}\Bigl\|\bar{A}^1 \bar{x}^1\Bigr\|^2
		&=
		  \frac{\rho}{2} \sum_{l=1}^n (\bar{x}^1_l)^\top (\bar{A}_l^1)^\top \bar{A}_l^1 \bar{x}^1_l
		\label{Eq:AppSecondTerm_2}
		\\
		&=
		  \frac{\rho}{2} \sum_{l=1}^n \sum_{p \in \mathcal{C}_1} D_{p,l} \Bigl( x_l^{(p)}\Bigr)^2
		\label{Eq:AppSecondTerm_3}
		\\
		&=
		  \frac{\rho}{2} \sum_{p \in \mathcal{C}_1} \sum_{l \in S_p} D_{p,l} \Bigl( x_l^{(p)}\Bigr)^2\,.
		\label{Eq:AppSecondTerm_4}
	\end{align}
	From~\eqref{Eq:AppSecondTerm_2} to~\eqref{Eq:AppSecondTerm_3} we just used the structure of~$\bar{A}_l^1$. Namely, if it exists, $(\bar{A}_l^1)^\top \bar{A}_l^1$ is a diagonal matrix, where each diagonal entry is extracted from the diagonal of~$A_l^\top A_l$, the Laplacian matrix for~$\mathcal{G}_l$. Since each entry in the diagonal of a Laplacian matrix contains the degrees of the respective nodes, the diagonal of $(\bar{A}_l^1)^\top \bar{A}_l^1$ contains $D_{p,l}$ for all~$p \in \mathcal{C}_1$. The reason why~$(\bar{A}_l^1)^\top \bar{A}_l^1$ is diagonal is because nodes with the same color are never neighbors. As in~\eqref{Eq:AppSecondTerm_1}, we exchanged the order of the summations from~\eqref{Eq:AppSecondTerm_3} to~\eqref{Eq:AppSecondTerm_4}.

	Finally, we develop the second term of~\eqref{Eq:AppSecondTerm_1}:
	\begin{align}
		  \rho &(\bar{A}^1\bar{x}^1)^\top \sum_{c=2}^C \bar{A}^c \bar{x}^{c,k}
		\notag
		\\
		&=
		  \rho \sum_{c=2}^C \sum_{l=1}^n (\bar{x}^1_l)^\top (\bar{A}^1_l)^\top (\bar{A}^c_l) \,\bar{x}_l^{c,k}
		\label{Eq:AppSecondTerm_5}
		\\
		&=
		  -\rho \sum_{c=2}^C \sum_{l=1}^n \sum_{p \in \mathcal{C}_1} \sum_{j \in \mathcal{N}_p \cap \mathcal{C}_c \cap \mathcal{V}_l} {x_l^{(p)}}^\top x_l^{(j),k}
		\label{Eq:AppSecondTerm_6}
		\\
		&=
		  -\rho \sum_{p \in \mathcal{C}_1} \sum_{l \in S_p} {x_l^{(p)}}^\top \sum_{c=2}^C \sum_{j \in \mathcal{N}_p \cap \mathcal{C}_c \cap \mathcal{V}_l} x_l^{(j),k}
		\label{Eq:AppSecondTerm_7}
		\\
		&=
		  -\rho \sum_{p \in \mathcal{C}_1} \sum_{l \in S_p}  \sum_{j \in \mathcal{N}_p \cap \mathcal{V}_l}  {x_l^{(p)}}^\top x_l^{(j),k}\,.
		\label{Eq:AppSecondTerm_8}
	\end{align}
	In~\eqref{Eq:AppSecondTerm_5} we just used the structure of~$\bar{A}^1$ and~$\bar{A}^c$, as visualized in~\eqref{Eq:PartialStructureMatrices}. From~\eqref{Eq:AppSecondTerm_5} to~\eqref{Eq:AppSecondTerm_6} we used the fact that $(\bar{A}^1_l)^\top \bar{A}^c_l$ is a submatrix of~$A_l^\top A_l$, the Laplacian of~$\mathcal{G}_l$, containing some of its off-diagonal elements. More concretely, $(\bar{A}^1_l)^\top \bar{A}^c_l$ contains the entries of~$A_l^\top A_l$ corresponding to all the nodes~$i \in \mathcal{C}_1 \cap \mathcal{V}_l$ and~$j \in \mathcal{C}_c \cap \mathcal{V}_l$. And, for such nodes, the corresponding entry in~$A_l^\top A_l$ is~$-1$ if~$i$ and~$j$ are neighbors, and~$0$ otherwise. From~\eqref{Eq:AppSecondTerm_7} to~\eqref{Eq:AppSecondTerm_8} we just used the fact that the set $\{\mathcal{C}_c\}_{c=2}^C$ is nothing but a partition of the set of neighbors of any node with color~$1$.
	Using~\eqref{Eq:AppFirstTerm_1}, \eqref{Eq:AppSecondTerm_1}, \eqref{Eq:AppSecondTerm_4}, and~\eqref{Eq:AppSecondTerm_8} in~\eqref{Eq:PartialC1_1}, we get~\eqref{Eq:PartialC1_2}.
	\hfill
	\qed

\end{document}